\documentclass[11pt]{article}
\usepackage{amssymb,amsbsy,amsmath,amsfonts,amssymb,amscd}
\usepackage{latexsym}
\newtheorem{lemma}{Lemma}[section]
\newtheorem{theorem}[lemma]{Theorem}
\newtheorem{corollary}[lemma]{Corollary}
\newtheorem{proposition}[lemma]{Proposition}
\newtheorem{definition}[lemma]{Definition}

\begin{document}

\def\C{{\mathbb C}}
\def\N{{\mathbb N}}
\def\Z{{\mathbb Z}}
\def\R{{\mathbb R}}
\def\PP{\cal P}
\def\p{\rho}
\def\phi{\varphi}
\def\ee{\epsilon}
\def\ll{\lambda}
\def\l{\lambda}
\def\a{\alpha}
\def\b{\beta}
\def\D{\Delta}
\def\g{\gamma}
\def\rk{\text{\rm rk}\,}
\def\dim{\text{\rm dim}\,}
\def\ker{\text{\rm ker}\,}
\def\square{\vrule height6pt width6pt depth 0pt}
\def\epsilon{\varepsilon}
\def\phi{\varphi}
\def\kappa{\varkappa}
\def\strl#1{\mathop{\hbox{$\,\leftarrow\,$}}\limits^{#1}}

\def\wz{\thinspace}
\def\proof{P\wz r\wz o\wz o\wz f.\hskip 6pt}
\def\quest#1{\hskip5pt {\scshape  Problem} {\rm #1}.\hskip 6pt}
\def\leq{\leqslant}
\def\geq{\geqslant}
\def\pd#1#2{\frac{\partial#1}{\partial#2}}
\def\limsup{\mathop{\overline{\hbox{\rm lim}\,}}}
\def\ug#1#2{\left\langle#1,#2\right\rangle}
\def\kk#1#2{{\k\langle#1,#2\rangle}}
\def\sv{\bf{ k} \langle X \rangle}
\def\k{k }
\def\lxr{\langle X \rangle}
\def\defin#1{\smallskip\noindent
{\scshape  Definition} {\bf #1}{\bf .}\hskip 8pt\sl}

\def\doubarr#1#2{\mathop{\hbox{$\vcenter{\offinterlineskip\halign
{\kern2pt\hfil##\hfil\kern2pt\cr \vrule height6pt width0pt
depth0pt\cr \smash{${\longleftarrow}\!\!{-}\!\!{-}$}\cr \vrule
height4pt width0pt depth0pt\cr
\smash{${\longleftarrow}\!\!{-}\!\!{-}$}\cr
}}$}}\limits_{#1}^{#2}}

\def\bull#1{\mathop{\bullet}\limits_{#1}}

\def\siglearr#1{\mathop{\hbox{${-}\!\!{-}\!\!{\longrightarrow}$}}\limits^{#1}}

\font\LINE=line10 scaled 1440 \font\Line=line10 \font\bIg=cmr10
scaled 1728 \font\biG=cmr10 scaled 1440 \font\LIne=line10
\def\lini{{\LINE\char"40}}
\def\li{{\Line\char"40}}
\def\lili{{\LIne\char"40}}
\def\pph{\vrule height4pt width0pt depth0pt}
\def\pha{{\phantom0}}
\def\sst{\scriptstyle}
\def\ddd{\displaystyle}
\def\>{\multispan}

\def\C{{\mathbb C}}
\def\N{{\mathbb N}}
\def\Z{{\mathbb Z}}
\def\R{{\mathbb R}}
\def\F{{\cal F}}
\def\U{{\cal U}}
\def\M{{\cal M}}
\def\Q{{\cal Q}}
\def\H{{\cal H}}
\def\E{{\cal E}}
\def\rr{{\cal R}}
\def\pp{{\cal P}}
\def\epsilon{\varepsilon}
\def\kappa{\varkappa}
\def\phi{\varphi}
\def\leq{\leqslant}
\def\geq{\geqslant}
\def\dim{\hbox{\rm dim}\,}
\def\ker{\hbox{\rm ker}\,}
\def\Cent{\hbox{\rm Cent}\,}
\def\ext{\hbox{\rm ext}\,}
\def\det{\hbox{\rm det}\,}
\def\deg{\hbox{\rm deg}\,}
\def\ssub#1#2{#1_{{}_{{\scriptstyle #2}}}}

\def\ramka#1#2#3{\hbox{$\vcenter{\offinterlineskip\halign
{\vrule\vrule\kern#2pt\hfil##\hfil\kern#2pt\vrule\vrule\cr
\noalign{\hrule}
\noalign{\hrule}
\vrule height #3pt depth0pt width0pt\cr
#1\cr
\vrule height #3pt depth0pt width0pt\cr
\noalign{\hrule}
\noalign{\hrule}}
}$}}

\def\jdots{\hbox{$\vcenter{\offinterlineskip\halign
{\hfil##\hfil\kern2pt&\hfil##\hfil\kern2pt&\hfil##\hfil\cr
&&.\cr
\noalign{\kern3pt}
&.&\cr
\noalign{\kern3pt}
.&&\cr
}}$}}

\def\Mone{\,\,\,\,\,\,\hbox{$\vcenter{\offinterlineskip\halign
{##\hfil&##\hfil&##\hfil&##\hfil&##\hfil&##\hfil&##\hfil\cr
&&&&&\smash{$\sst\bullet$}&$\sst n-1$\cr
\lini&\lini&\lini&&&\smash{$\pha\atop\ddd\vdots$}&\cr
\lini&\lini&\lini&\lini&&&\cr
\lini&\lini&\lini&\lini&\lini&&\smash{${}_2$}\cr
&\lini&\lini&\lini&\lini&&\smash{${}_1$}\cr
&&\lini&\lini&\lini&&\smash{${}_0$}\cr
\smash{$\sst\bullet$}&\rlap{\ $\dots$}&&&&&\cr
\vrule height 3pt depth0pt width0pt\cr&&&&&&\cr
\llap{$\sst-n$}$\sst+1$&&\rlap{$\,\,\,\sst-2$}&\rlap{$\,\,\,\sst-1$}&&&\cr
}}$}\,}

\def\MtwoA{\hbox{$\vcenter{\offinterlineskip\halign
{##\hfil&##\hfil&##\hfil&##\hfil&##\hfil\cr
0&\ \ 1&&\ \smash{\lower 6pt\hbox{\biG0}}&\cr
\noalign{\kern-2pt}
&$\,\ddots$&$\ddots$&&\cr
\noalign{\kern-6pt}
&&$\ddots$&$\ddots$&\cr
\noalign{\kern-6pt}
&\smash{\biG0}&&$\ddots$&\smash{\raise 6pt\hbox{\ 1}}\cr
\noalign{\kern6pt}
&&&&\smash{\raise 1pt\hbox{\ 0}}\cr
}}$}}

\def\MtwoB{\small \,\hbox{$\vcenter{\offinterlineskip\halign
{##\hfil&##\hfil&##\hfil&##\hfil&##\hfil&##\hfil\cr
&&&\rlap{\large$\dots$}&&\smash{$\sst\bullet$}\cr
\lini&\rlap{\lini}\kern.5pt\rlap{\lini}\kern.5pt\lini\kern-1pt&\lini&&&{\large\smash{$\pha\atop\ddd\vdots$}}\cr
&\lini&\rlap{\lini}\kern.5pt\rlap{\lini}\kern.5pt\lini\kern-1pt&\lini&&\cr
&&\lini&\rlap{\lini}\kern.5pt\rlap{\lini}\kern.5pt\lini\kern-1pt&\lini&\cr
&\smash{\bIg0}&&\lini&\rlap{\lini}\kern.5pt\rlap{\lini}\kern.5pt\lini\kern-1pt&\cr
&&&&\lini&\cr
}}$\normalsize}}

\def\Mfour{\,\hbox{$\vcenter{\offinterlineskip\halign
{\hfil##&##&##&##\hfil\cr
\ramka{$J_1$}{7}{7}&&\ \bIg0\cr
&\ramka{$J_2$}{10}{10}&&\cr
\noalign{\kern-5pt}
\smash{\lower 5pt\hbox{\bIg0\ }}&&$\ddots$&\cr
&&&\ramka{$J_m$}{4}{5}\cr
}}$}\,}

\def\Mfive{\,\hbox{$\vcenter{\offinterlineskip\halign
{\vrule\hfil##\hfil\vrule&\hfil##\hfil\vrule&\hfil##\hfil\vrule&\hfil##\hfil\vrule\cr
\noalign{\hrule}
\ramka{$A_{11}$}{9}{12}&$A_{12}$&&$A_{1m}$\cr
\noalign{\hrule}
&\ramka{$A_{22}$}{12}{15}&&\cr
\noalign{\hrule}
\noalign{\kern-5pt}
&&$\ddots$&\cr
\noalign{\hrule}
$A_{m1}$&$A_{m2}$&&\ramka{$A_{mm}$}{4}{10}\cr
\noalign{\hrule}
}}$}\,}

\def\Msix{\,\hbox{$\vcenter{\offinterlineskip\halign
{\vrule\hfil##\hfil\vrule&\hfil##\hfil\vrule&\hfil##\hfil\vrule&\hfil##\hfil\vrule\cr
\noalign{\hrule}
\ramka{$[A_{11}]$}{9}{14}&$[A_{12}]$&&$[A_{1m}]$\cr
\noalign{\hrule}
&\ramka{$[A_{22}]$}{12}{17}&&\cr
\noalign{\hrule}
\noalign{\kern-5pt}
&&$\ddots$&\cr
\noalign{\hrule}
$[A_{m1}]$&$[A_{m2}]$&&\ramka{$[A_{mm}]$}{4}{12}\cr
\noalign{\hrule}
}}$}\,}

\vfill\break



\def\MsevenA{\,\hbox{$\vcenter{\offinterlineskip\halign
{\vrule\kern45pt##\hfil&##\hfil&##\hfil&##\hfil&##\hfil&##\hfil\vrule\cr
&&&&&\rlap{\smash{\lower2pt\hbox{\kern-2pt$\sst\bullet$}}}\cr
\noalign{\hrule}
\lili&\lili&\lili&&\rlap{\kern-4pt\smash{$\jdots$}}&\cr
&\lili&\lili&\lili&&\cr
\kern-30pt\smash{\bIg0}&&\lili&\lili&\lili&\cr &&&\li&\li&\cr
&&&&\lili&\cr \noalign{\hrule} }}$}\,}

\def\MsevenB{\,\hbox{$\vcenter{\offinterlineskip\halign
{\vrule##\hfil&##\hfil&##\hfil&##\hfil&##\hfil&##\hfil\vrule\cr
&&&&&\rlap{\smash{\lower2pt\hbox{\kern-2pt$\sst\bullet$}}}\cr
\noalign{\hrule}
\lili&\lili&\lili&&\rlap{\kern-4pt\smash{$\jdots$}}&\cr
&\lili&\lili&\lili&&\cr &&\lili&\lili&\lili&\cr &&&\lili&\lili&\cr
&&\smash{\lower30pt\hbox{\bIg0}}&&\lili&\cr \vrule height45pt
depth0pt width0pt &&&&&\cr \noalign{\hrule} }}$}\,}

\def\Meighta{\hbox{$\vcenter{\offinterlineskip\halign
{\vrule\vrule##\hfil&##\hfil&##\hfil&##\hfil\vrule\vrule\cr
\noalign{\hrule}
\noalign{\hrule}
\li&\rlap{\li}\kern.4pt\rlap{\li}\kern.4pt\li\kern-.8pt&\li&\cr
&\li&\rlap{\li}\kern.4pt\rlap{\li}\kern.4pt\li\kern-.8pt&\li\cr
&&\li&\rlap{\li}\kern.4pt\rlap{\li}\kern.4pt\li\kern-.8pt\cr
&&&\li\cr
\noalign{\hrule}
\noalign{\hrule}
}}$}}

\def\Meightaa{\hbox{$\vcenter{\offinterlineskip\halign
{##\hfil&##\hfil&##\hfil&##\hfil\cr
\li&\li&\li&\cr
&\li&\li&\li\cr
&&\li&\li\cr
&&&\li\cr
}}$}}

\def\Meightb{\hbox{$\vcenter{\offinterlineskip\halign
{\vrule\vrule##\hfil&##\hfil&##\hfil&##\hfil&##\hfil\vrule\vrule\cr
\noalign{\hrule}
\noalign{\hrule}
\li&\rlap{\li}\kern.4pt\rlap{\li}\kern.4pt\li\kern-.8pt&\li&&\cr
&\li&\rlap{\li}\kern.4pt\rlap{\li}\kern.4pt\li\kern-.8pt&\li&\cr
&&\li&\rlap{\li}\kern.4pt\rlap{\li}\kern.4pt\li\kern-.8pt&\li\cr
&&&\li&\rlap{\li}\kern.4pt\rlap{\li}\kern.4pt\li\kern-.8pt\cr
&&&&\li\cr
\noalign{\hrule}
\noalign{\hrule}
}}$}}

\def\Meightc{\hbox{$\vcenter{\offinterlineskip\halign
{\vrule\vrule##\hfil&##\hfil&##\hfil\vrule\vrule\cr
\noalign{\hrule}
\noalign{\hrule}
\li&\rlap{\li}\kern.4pt\rlap{\li}\kern.4pt\li\kern-.8pt&\li\cr
&\li&\rlap{\li}\kern.4pt\rlap{\li}\kern.4pt\li\kern-.8pt\cr
&&\li\cr
\noalign{\hrule}
\noalign{\hrule}
}}$}}

\def\Meightcc{\hbox{$\vcenter{\offinterlineskip\halign
{##\hfil&##\hfil&##\hfil\cr
\li&\li&\li\cr
&\li&\li\cr
&&\li\cr
}}$}}

\def\Meight{\,\,\hbox{$\vcenter{\offinterlineskip\halign
{\vrule\hfil##\vrule&\hfil##\vrule&\hfil##\hfil\vrule&\hfil##\vrule\cr
\noalign{\hrule}
\Meighta\llap{\lower7pt\hbox{0}\kern23pt}&\Meightaa\llap{\lower7pt\hbox{0}%
\kern33pt}&&\raise5.5pt\Meightcc\llap{\lower7pt\hbox{0}\kern18pt}\cr
\noalign{\hrule}
\raise5pt\Meightaa\llap{\lower7pt\hbox{0}\kern23pt}&\Meightb\llap{\lower7pt%
\hbox{0}\kern33pt}&&\raise10pt\Meightcc\llap{\lower7pt\hbox{0}\kern18pt}\cr
\noalign{\hrule}
&&\raise5pt\hbox{$\,\ddots\,$}&\cr
\noalign{\hrule}
\Meightcc\llap{\lower5pt\hbox{0}\kern23pt}&\Meightcc\llap{\lower5pt%
\hbox{0}\kern33pt}&&\Meightc\llap{\lower5pt\hbox{0}\kern18pt}\cr
\noalign{\hrule}
}}$}\,\,}


\vskip1cm

\title{Representation spaces of the Jordan plane}

\author{
 Natalia K. Iyudu}

\date{}

\maketitle

\small

\centerline{Max-Planck-Institut F\"ur Mathematik, Vivatsgasse 7, 53111 Bonn, Germany }

\smallskip

\centerline{Department of Pure Mathematics, Queen's University
Belfast, Belfast BT7 1NN, U.K.}

\smallskip

\centerline{ {\bf e-mail:}  \,\, n.iyudu@qub.ac.uk}

\bigskip

\small

{\bf Abstract}

\bigskip

We investigate relations between the properties of an algebra and
its varieties of finite-dimensional module structures, on the example
of the Jordan plane $R=k\langle x,y\rangle/ (xy-yx-y^2)$.

Complete description of irreducible components of the
representation variety $mod (R,n)$  obtained for any dimension $n$,
it is shown that the variety is equidimensional.

The influence of the property of the non-commutative Koszul (Golod-Shafarevich)
complex to be a DG-algebra resolution of an algebra (NCCI), on the structure of representation spaces is studied.
It is shown that the Jordan plane provides a new example of RCI (representational complete intersection).

\vspace{5mm}



\normalsize

\bigskip
\tableofcontents
\medskip

\section{Introduction}

We investigate how properties of a quadratic algebra are reflected in the properties of its finite dimensional representations. Our main example is the prominent since the  foundation of noncommutative geometry \cite{AS},\cite{SM} quadratic algebra
$R=k\langle x,y\rangle/ (xy-yx-y^2)$, known as the Jordan plane.
According to the Artin-Schelter classification  \cite{AS} there are three types of algebras of global dimension two: the usual quantum plane $k\langle x,y\rangle/ (xy-qyx)$, for $q \neq 0$ and the Jordan plane
are regular algebras, and one more, which is not: $k\langle x,y\rangle/ (xy)$,

The Jordan plane appeared also in many different contexts in mathematics and physics \cite{Ma,Ko,a1,ia,wia,iw,cs}.
To mention a few connections to other important objects,
let us note
that  $R$ is a subalgebra of the first Weyl algebra $A_1$. The
latter has no finite dimensional representations, but $R$ turns out
to have quite a rich structure of them. Category of finite
dimensional modules over $R$ contains, for example, as a full
subcategory ${\rm mod} \, GP(n,2)$, where $GP(n,2)$ is the
Gelfand--Ponomarev algebra \cite{GP} with the nilpotency degrees of
variables $x$ and $y$,  $n$ and $2$ respectively.

We prove that the space of $n$-dimensional representations of the Jordan plane, which we call the {\it Jordan variety}, is equidimensional and irreducible components are parametrised by partitions of $n$.
The description of irreducible components
is given for any $n$. This result is obtained on the basis of a lucky choice of stratification of the representation space. The stratum $\cal U(P)$ is defined as those pairs of matrices, where the matrix $Y $ (the image of the variable $y$ under the representation) has the fixed Jordan form.  Closures of chosen in such a way strata provide a complete list of irreducible components of the corresponding variety.
As a consequence, irreducible components parametrised by partitions of $n$, defined  by the block structure of the Jordan form of $Y$, the eigenvalues of $Y$ are not involved, since in any representation this matrix is nilpotent.

In each stratum there is  a vector bundle structure, such that the dimension of the base is  $m$ and the dimension of the fiber is $n^2-m$, so the dimension of each stratum  is $n^2$.  Since irreducible components of the variety are precisely closures of these strata, we can see that variety of $n$-dimensional representations is equidimensional for any $n$.

Thus the following theorem is proved in the paper.

{\bf Theorem 12.1} {\it Any irreducible component $K_j$ of the
representation variety $mod(R,n)$ of the Jordan plane contains only
one stratum $U_{\cal P}$ from the stratification related to the Jordan
normal form of $Y$, and is the closure of this stratum.

The number of irreducible components of the variety $mod(R,n)$ is
equal to the number of the partitions of $n$.

The variety  $mod(R,n)$ is equidimensional and the dimension of components is equal to $n^2$.}

\vspace{5mm}

The {\it generic} situation in this variety is provided by the component corresponding to the full Jordan block $Y$. We study it in details, especially since the
further results on the structure of representation variety for $R$ show an
exceptional role of such strata, for example, they turn out to be the only building blocks in the
analogue of the Krull-Remark-Schmidt decomposition theorem on the level of irreducible components.
The following holds.

{\bf Corollary}  {\it Only the irreducible component
$K_{(n)}=\overline {{\cal U}_{(n)}}$ which is the closure of the
stratum corresponding to the trivial partition of $n$ (the full block
$Y$) contains an open dense subset consisting of indecomposable
modules}.

Another evidence of the special role of this 'full block' stratum presented in section
\ref{Ger} where we estimate dimensions of image algebras. In section
\ref{Ger} we prove an analogue of the Gerstenhaber--Taussky--Motzkin
theorem \cite{Ger}, \cite{Gur}, \cite{MT} on the dimension of algebras
generated by two commuting matrices. We show that the dimension of
image algebras of representations of $R$ does not exceed $n(n+2)/4$
for even $n$ and $(n+1)^{2} /4$ for odd $n$, and this estimate is
attained in the stratum related the full Jordan block  $Y$.

In the section \ref{ring} we perform a study of the generic strata in terms of the properties of algebras-images. Namely, for algebras-images from the strata with full-block $Y$,  we clarify,
 using the Ringel classification of local complete algebras  \cite{Ringloc}, that they are all {\it tame} for $n \leq 4$
and {\it wild} for $n \geq 5$. The main tool in this section is to solve the problem of finding defining relations for algebra-image of finite dimensional representation. When it is done, and we have precisely written relations, we compare them with the Ringel list of local complete algebras, to determine whether they are tame or wild.

Referring to results due to Kontsevich \cite{Maxim}, concerning the existence of the {\it algebraic} bundle
on the generic subvariety in the space of all pairs of $n \times n$ matrices ${\cal M}_n$; to the well-known results
on the {\it commuting variety}, where on the generic subvariety (consisting of pairs of matrices, one of which is nonderogatory)
there is a vector bundle with the base, consisting of the set of nonderogatory matrices,
of dimension $n^2$, and the $n$-dimensional fiber, consisting of all polynomials on the matrix; we guess
that in any closed subvariety of  ${\cal M}_n$ given by a (homogeneous) quadratic matrix equation, in general position there is a natural structure of algebraic (or sometimes even a vector) bundle.

Another main subject of our attention is to investigate, how certain properties of Golod-Shafarevich (noncommutative Koszul) complex, like property of algebra of being NCCI ({\it noncommutative complete intersection}) could be reflected on the level of representation varieties. The algebras which are NCCI
fall into the class of Golod-Shafarevich algebras, which is defined  in [EF] by E.Zelmanov and have been studied in \cite{Ag}, \cite{er}, \cite{Ag1},\cite{Ag2},\cite{TV} etc.
We ensure that Jordan plane is a NCCI and derive from the previous results on representation spaces that (for infinitely many $n$), spaces of $n$-dimensional representations of Jordan plane are complete intersections. By this we establish a new example of RCI ({\it representational complete intersection}), the notion introduced by Ginzburg and Etingof in \cite{GE}. In this paper the main source of examples of RCIs came from the preprojective algebras of finite quivers. The general interrelation between those two properties also was clarified, namely it was shown that RCI implies NCCI (in the above sense),
 so the restriction on algebra formulated in terms of properties of representation spaces (RCI)is quite strong.

Let us note, that sections 1-9 can be considered as a preparatory
 part of the paper, where we establish various properties  of Jordan plane itself and
 algebras-images of finite dimensional representations.  These include facts that $R$ is Koszul,
 multiplication formulas in canonical basis, description of  prime, primitive ideals and automorphisms, the fact that $R$ is residually finite dimensional, as well as description of
 irreducible modules,   some properties of  indecomposable modules, description of all modules in terms of ${\cal B}-Toeplitz$ matrices. These facts were established in our earlier work \cite{MP}, and  will serve here as a basis for obtaining further results.

\section{Structural properties of the images of
representations and quivers}\label{quiv}

We consider here the case when $k$ is an algebraically closed field
of characteristic zero. Let $\p: R \rightarrow {\rm End} (k^n)$ be
an arbitrary finite dimensional representation of $R$, denote by
$A_{\p,n}=\p (R)$ an image of $R$ in the endomorphism ring. We will
write also $A_n$ or $A$ when it is clear from the context which $\p$
and $n$ we mean.

We derive in this section some main structural properties   of algebras
$A_{n,\p}$. They all are basic algebras,
the image of $y$ in any such algebra is nilpotent, complete system of orthogonal
idempotents in $A$ is defined by the set of different eigenvalues
of $\p(x)$, etc.

Lemmata below describe the  structural properties of image algebras
for $R$.

The following fact certainly allows many different proofs, we
present here the shortest we know.

\begin{lemma}\label{l2.1} Let $Y=\rho_n(y)$. Then the matrix $Y$ is nilpotent.
\end{lemma}

\proof Suppose that the matrix $Y$ is not nilpotent and hence has a
nonzero eigenvalue $z$. We take the eigenprojector $P$ on the eigenspace $E_z=ker(Y-zI)^n$,
corresponding to this eigenvalue. It is obviously commutes with  $Y: PY=YP$ and is an idempotent operator:
$P^2=P$. Hence multiplying our relation $XY-YX=Y^2$ from the right
and from the left hand  side by $P$ and using the above two facts, we
can observe that operators $X'=PXP$ and $Y'=PYP$ also satisfy the
same relation: $X'Y'-Y'X'={Y'}^2$. Taking into account that $Y'$ has only
one nonzero eigenvalue $z$, we see that traces of right and left hand sides of the relation
can not coincide. Indeed, $tr(X'Y'-Y'X')=0 $ and $tr Y'^2=z^2 dim E_z \neq 0$. This contradiction completes the proof. $\Box$

Let us prove also a little bit more general fact of linear algebra.

\begin{lemma} \label{nilnil}
Let $X,Y$ be $n\times n$ matrices over an algebraically closed field
$k$ of characteristic zero. Assume that the commutator $Z=XY-YX$
commutes with $Y$. Then $Z$ is nilpotent.
\end{lemma}

\proof
Assume the contrary.
Then $Z$ has a non-zero eigenvalue $z\in k$. Let
$$
L={\rm ker}\,(Z-zI)^n\ \ \text{and}\ \ N={\rm Im}(Z-zI)^n.
$$
The subspace $L$ is known as a main subspace for $Z$ corresponding
to the eigenvalue $z$. Clearly $L\neq \{0\}$. It is well-known that
the whole space $k^n$ can be split to a direct sum $L\oplus N$ of $Z$-invariant linear
subspaces $L$ and $N$. Due to $ZY=YZ$, the subspaces $L$ and $N$ are
also invariant for $Y$.

Consider the linear projection $P$ along $N$ onto $L$. Since $L$ and
$N$ are invariant under both $Y$ and $Z$, we have $ZP=PZ$ and
$YP=PY$. Multiplying the equality  $Z=XY-YX$ by $P$ from the left
and from the right hand side and using the equalities $ZP=PZ$,
$YP=PY$ and $P^2=P$, we get
$$
ZP=PXPY-YPXP.
$$
Since $ZP$ vanishes on $N$ and $ZP|_L$ has only one eigenvalue $z$,
then after restriction to $L$, we have ${\rm tr}\,ZP=z\,{\rm
dim}\,L$. On the other hand ${\rm tr}\,PXPY={\rm tr}\,YPXP$ since
the trace of a product of two matrices does not depend on the order
of the product. Thus, the last display implies that
$
z\,{\rm dim}\,L=0,
$
which is not possible since $z\neq 0$ and ${\rm dim}\,L>0$. $\Box$

Coming back to the case of the Jordan plane, we have further

\begin{lemma}\label{l2.2} Let $X=\rho_n(x)$ and $\{\lambda_1,...,\lambda_r    \}= {\rm Spec} X$. Then the matrix
$S=S(X)=(X-\lambda_1 I) \dots (X-\lambda_r I)$ is nilpotent.
\end{lemma}

\proof Note that Spec$\,p(X)=p({\rm Spec}\,X)$ for any polynomial
$p$. Spec$\,X$ in our case is $\{\lambda_1,\dots,\lambda_r\}$ and
hence Spec$\,S=\{0\}$. Therefore the matrix $S$ is nilpotent. $\Box$

Let $J(A)=J$ be the Jacobson radical of the algebra $A_{\p,n}$.

\begin{lemma}\label{l2.3} Any nilpotent element of the algebra $A=\rho(R)$
belongs to the radical $J(A)$.  \end{lemma}

\proof We will use here the feature of an algebra $A$ that it has
the presentation as a quotient of the free algebra, containing our main
relation. Namely, it has a presentation: $A=\k\langle x,y|
xy-yx=y^2, R_A \rangle$, where $R_A \subset \k\langle x,y \rangle$
is the set of additional relations specific for the given image
algebra. Thus we can think of elements in $A$ as of polynomials in
two variables (subject to some relations). Let $Q(x)$ be a
polynomial on one variable $Q(x)\in \k[x]$ and $Q(X)\in A$ be a
nilpotent element with the degree of nilpotency $N$: $Q^N=0$. We
show first that $Q\in J(A)$. We have to check that for any
polynomial $a\in \k\langle x,y\rangle$, $1-a(X,Y)Q(X)$ is
invertible. It suffices to verify that $a(X,Y)Q(X)$ is nilpotent. By
Lemma \ref{l2.1} $Y$ is nilpotent. Denote by $m$ the degree of
nilpotency of $Y$: $Y^m=0$. Let us verify that
$(a(X,Y)Q(X))^{mN}=0$. Present $a(X,Y)$ as $u(X)+Yb(X,Y)$. If then
we consider a word of length not less then $mN$ on letters
$\alpha=u(X)Q(X)$ and $\beta=Yb(X,Y)Q(X)$ then we can see that it is
equal to zero. Indeed, if there are at least $m$ letters $\beta$
then using the relation $XY-YX=Y^2$ to commute the variables one can
rewrite the word as a sum of words having a subword $Y^m$. Otherwise
our word has the subword $\alpha^N=u(X)^NQ(X)^N=0$. Thus, $Q(X)\in
J(A)$.

Note now that if we have an arbitrary nilpotent polynomial
 $G(X,Y)$, we can separate the terms containing $Y$:
 $G(X,Y)=Q(X)+YH(X,Y)$. To obtain nilpotency of any element
 $a(X,Y)G(X,Y)$ it suffices to verify nilpotency of
 $a(X,Y)Q(X)$, which was already proven, because the
 relation $[X,Y]=Y^2$ allows to commute with $Y$,
 preserving (or increasing) the degree of $Y$. $\Box$

Since $A= \rho(R)$ is finite dimensional, we have:
\begin{corollary}\label{c2.1} The Jacobson radical of $A=\rho(R)$ consists
precisely of all nilpotent elements. \end{corollary}

In particular,

\kern-3mm

\begin{corollary}\label{c2.2} Let $Y=\rho(y)$. Then $Y\in J(A)$. \end{corollary}

Let us formulate here  another property of the radical, which
will be on use later on.

\begin{corollary}\label{radft}
The Jacobson radical of $A=\rho(R)$ consists of all polynomials on
$X=\p(x)$ and $Y=\p(y)$ without constant term if and only if $X$ is
nilpotent in $A$.
\end{corollary}

\proof In one direction this is trivial. We should ensure   only
that if $X^N=0$ then $p(X,Y)^{2N}=0$ for any polynomial $p$ such
that $p(0,0)=0$, using the relation $XY-YX=Y^2$, which is an easy
check.$\Box$

\begin{lemma}\label{t2.1} Let $A_{\rho,n}$ be  as above,
and $X=\rho_n(x)$, $Y=\rho_n(y)$ be its generators. Then
$A/J$ is a commutative one-generated ring $\k[x]/S(x)$, where
$S(x)=(x-\lambda_1)\dots(x-\lambda_k)$, and
$\lambda_1,\dots,\lambda_k$ are all different eigenvalues of the
matrix $X$. \end{lemma}

\proof From Corollary \ref{c2.2} we can see that $A/J$ is an
algebra with one generator $x$: $A/J \simeq \k[x]/I$. We are going to
find now an element which generates  the ideal $I$.

First of all by lemmas \ref{l2.2} and  \ref{l2.3} $S\in J(A)$, hence
$S(x)=(x-\lambda_1) \dots (x-\lambda_r) \in I$. Let us show now that
$S$ divides any element in $I$. If some polynomial $p\in \k[x]$ does
not vanish in some eigenvalue $\lambda$ of $X$ then $p(X)\notin
J(A)$. Indeed, the matrix $p(X)$ has a non-zero eigenvalue, than
$p(\lambda) \neq 0$ and hence $I-\frac 1{p(\lambda)}p(X)$ is
non-invertible. Therefore $p(X)\notin J(A)$. Thus, $S(x)$ is the
generator of $I$. This finishes the proof. $\Box$

\begin{corollary}\label{ct2.2} The system
$e_i=p_i(X)/p_i(\ll_i)$, where
$$p_i(X)=(X-\ll_1I)\dots\widehat{(X-\ll_iI)}\dots(X-\ll_rI)$$ and
$\ll_i$ are different eigenvalues of $X=\p(x)$ is a complete
system of orthogonal idempotents of $A/J$.
\end{corollary}

\proof Orthogonality of $e_i$ is clear from the presentation of
$A/J$ as $\k[x]/{\rm id}(S)$ proven in Lemma \ref{t2.1}. $\Box$

\begin{proposition}\label{t2.3}  For any finite dimensional representation
 $\p$ the semisimple part of $A_{\p}$ is a
product of a finite number of copies of the field $\k$:
$$ A/J = \prod_{i=1}^r \k_i, $$ where $r$ is the number of
different eigenvalues of the matrix $X=\rho(x)$.\end{proposition}

\proof We shall construct an isomorphism between $A/J$ and $\displaystyle
\prod_{i=1}^r \k_i$ using the system $e_i$, $i=1,\dots,r$ of
idempotents from Corollary  \ref{ct2.2}. Clearly $e_i$ form a basis
of $A/J$ as a linear space over $\k$. From the presentation of $A/J$
as a quotient $\k[x]/{\rm id}(S)$ given in lemma \ref{t2.1} it is
clear that the dimension of $A/J$ is equal to the degree of
polynomial $S(x)$, which coincides with the number of different
eigenvalues of the matrix $X$. Since idempotents $e_i$ are
orthogonal, they are linearly independent and therefore form a basis
of $A/J$. The multiplication of two arbitrary elements $a,b\in A/J$,
$a=a_1e_1+{\dots}+a_r e_r$, $b=b_1e_1+{\dots}+b_r e_r$ is given by
the formula $ab=a_1b_1e_1+{\dots}+a_r b_r e_r$ due to orthogonality
of the idempotents $e_i$. Hence the map $a\mapsto(a_1,\dots,a_r)$ is
the desired isomorphism of $A/J$ and $\displaystyle \prod_{i=1}^r
\k_i$. $\Box$

Since all images are basic algebras we can associate to each of
them a {\it quiver} in a conventional way (see, for example,
\cite{Gr}, \cite{Kirich}).

The vertices will correspond to the idempotents $e_i$ or, by
Corollary \ref{ct2.2}, equivalently, to the different eigenvalues of
matrix $X$. The number of arrows from vertex $e_i$ to the vertex
$e_j$ is the ${\rm dim}_\k \, e_i(J/J^2)e_j$. There are a finite
number of such quivers in fixed dimension $n$ (the number of
vertices bounded by $n$, the number of arrows between any two
vertices roughly by $n^2$).

We can define an equivalence relation on representations of algebra
$R$ using quivers of their images.

\begin{definition} Two representations $\rho_1$ and $\rho_2$ of the
algebra $R$ are quiver-equivalent $\rho_1 \sim_{Q} \rho_2$ if
quivers associated to algebras $\rho_1(R)$ and $\rho_2(R)$ coincide.
\end{definition}

This will lead to a rough classification of representations by means of quivers of their images.

{\bf Toy example.}
As an example let us clarify the question on how many
quiver-equivalence classes appear in the family of representations
$$
{\cal U}_{(n)}=\{(X,Y)\in {\rm mod}\,(R,n)|{\rm rk}\,Y=n-1\}
$$
and which quivers are realized.

\begin{proposition}\label{t2.4}For any $n \geq 3$ families of representations ${\cal U}_{(n)}$ belong
to one quiver-equivalence class. Corresponding quiver consists of
one vertex and two loops. \end{proposition}

Let us ensure first the following lemma.

\begin{lemma}\label{radrsq}
If in the representation $\p: R \rightarrow A, \, \,\, X=\p(x)$ has
only one eigenvalue $\lambda$, then the corresponding quiver $Q_A$
has one vertex and number of loops is a dimension of the vector
space ${\rm Span}_k\{ \bar X-\lambda I, \bar Y\}$, which  does not
exceed 2. Here $\bar Y=\phi Y$ and  $\bar X=\phi X$  for $\phi:A \rightarrow A/J^2$.
\end{lemma}

\proof Due to the description of idempotents above, in the case of
one eigenvalue the only idempotent is unit. Hence we have to
calculate ${\rm dim}_k J/J^2$, where $J=Jac(A)$. Since $X-\lambda I$
satisfies the same relation as $X$, we could apply corollary
\ref{radft} and result immediately follows. $\Box$

\proof(of Proposition \ref{t2.4})
This will directly follow from Lemma \ref{radrsq}, after we
show in section \ref{descript} that $X$ has only one eigenvalue in
the family ${\cal U}_{(n)}$ and take into account that when we have
full block $Y$, the dimension of the linear space ${\rm Span}_k\{ \bar
X-\lambda I, \bar Y\}$ can not be smaller then 2. $\Box$

\section{Jordan Calculus}

Here we shall prove lemmata containing  formulas for multiplication in the canonical linear basis of the Jordan plane, it will be on
use for various purposes later on.

\subsection{Gr\"obner basis of the ideal and a linear basis of algebra}

 The basis of our algebra as a vector space over $k$ consists of the
monomials $y^kx^l$, $k,l=0,1,\dots$. Let us ensure this in a canonical way.

For this we
remind  the definition of a Gr\"obner  basis of an ideal and the
method of construction of a linear basis of an algebra given by
relations, based on the Gr\"obner basis technique.
Using this
canonical method it could be easily shown that, for example, some
Sklyanin algebras enjoys a PBW property. This was proved in
\cite{od}, the arguments there are very  interesting
in their own right, but quite involved.

Let $A=k\langle X\rangle/I$. The first essential step is to fix an
ordering on the semigroup $\ss=\langle X \rangle$. We fix some
linear ordering in the set of variables $X$ and extend it to an
{\it admissible} ordering on $\ss$, i.e. extend it in a way to satisfy the
conditions:

1) if  $u,v,w \in \ss$ and $ u<v$ then $uw<vw$ and $wu<wv$;

2) the descending chain condition (d.c.c.):
 there is no infinite properly
descending chain of elements of $\ss$.

We shall use the {\it deg\-ree-lex\-i\-c\-o\-g\-r\-a\-p\-h\-i\-cal}
ordering in the semigroup $\ss$, namely  for arbitrary $ \, u=
x_{i_1}\ldots x_{i_n}, v= x_{j_1}\ldots x_{j_k}\in\ss \quad \text{we
say} \quad
 u>v, \, \text{when either}\quad$ $\deg\, u > \deg\, v
$ $ \quad \text{or} \quad   \deg\, u =  \deg\, v \quad \text{and for
some} \,\, l: \,\, x_{i_l}>x_{j_l} \quad \text{and}\quad
x_{i_m}=x_{j_m}  \, \text{for any} \,\,  m<l.$
  This ordering is admissible.

Denote by $\bar f$ the highest term of polynomial $ f\in A=k\langle
X\rangle$ with respect to the above order.

{\bf Definition 4.2.} Subset $G \in I, I \triangleleft k \langle X
\rangle$ is a {\it Gr\"obner basis} of an ideal $I$ if the set of
highest terms of elements of $G$ generates the ideal of highest
terms of $I: id\{\bar G\} = \bar I$.

{\bf Definition 4.3.} We say that a monomial $u \in \langle X
\rangle$ is {\it normal} if it does not contain as a submonomial any
highest term of an element of the ideal I.

From these two definitions it is clear that normal monomial is a
monomial which does not contain  any highest term of an element of
Gr\"obner basis of the ideal $I$. If Gr\"obner basis turns out to
be finite then the set of normal words is constructible.

In the case when an ideal $I$ of defining relations for $A$ has a finite
Gr\"obner basis, the algebra called {\it standardly  finitely
presented(s.f.p.)}.

It is an easy, but useful fact that $\langle X \rangle$ is isomorphic
to the direct sum $I \oplus {{\langle N \rangle}_{k}} $ as a linear
space over $k$, where $\langle N \rangle_{k}$ is the linear span of
the set of normal monomials from $\lxr$ with respect to the ideal
$I$.  Hence given a Gr\"obner
basis $G$ of an ideal $I$, we can construct a linear basis of an
algebra $A=\langle X \rangle / I$ as a set of normal (with respect
to $I$) monomials, at least in case when $A$ is s.f.p.

As a consequence we immediately get the following

\begin{lemma}\label{l4.1}  The system of monomials $y^n x^m$ form a basis of
algebra $R$ as a vector space over $k$.\end{lemma}

\proof With respect to the ordering $x > y$ on the set of generators, and corresponding degree-lexicographic ordering on monomials $\langle X \rangle$, the relation $xy-yx-y^2$ forms a Gr\"obner basis of the ideal generated by itself.
This follows from the Bergman Diamond Lemma \cite{Berg}, which says, that the set of generators of an ideal form a Gr\"obner basis if all  ambiguities between highest terms of this set are solvable. We say that two monomials $u,v \in \langle X \rangle$ form an {\it ambiguity}, if there exists a monomial $w \in  \langle X \rangle$, such that $w=au=bv$ for some $a,b \in \langle X \rangle$. In our case the only highest term, the monomial $xy$ does not form any ambiguity with itself, so the set of ambiguities is empty.

\subsection{Multiplication formulas in the Jordan plane}\label{mult}

We say that an element is in the {\it normal form}, if
it is presented as a linear combination of normal monomials.

After we have a linear basis of normal monomials we should know how
to multiply them to get again an element in the normal form.

Now we are going to prove  the following lemmata, where we express
precisely the normal forms of some products.

\begin{lemma}\label{l3.1.1} The normal form of the monomial $xy^n$ in algebra $R$ is the following:

  $xy^n=y^nx+ny^{n+1}$.\end{lemma}

\proof This can be proven by induction on $n$. The case $n=1$ is
just the defining relation of our algebra. Suppose $n>1$ and the equality
$xy^{n-1}=y^{n-1}x+(n-1)y^{n}$ holds. Multiplying it by $y$ from
the right and reducing the result  by the relation $xy-yx=y^2$, we obtain
$$
xy^n=y^{n-1}xy+(n-1)y^{n+1}=y^nx+y^{n+1}+(n-1)y^{n+1}=y^nx+ny^{n+1}.
$$
The proof is now complete.
$\Box$

\begin{lemma}\label{l3.1.2}
 The normal form of the monomial $x^ny$ in algebra $R$ is the following:

$x^ny=\sum\limits_{k=1}^{n+1}\alpha_{k,n}y^kx^{n-k+1}$, where
$\alpha_{k,n}= n! / (n-k+1)! $ for $ k=1,...,n+1$.
\end{lemma}

\proof  We are going to prove this formula inductively using the
previous lemma. As a matter of fact we shall obtain recurrent
formulas for $\alpha_{k,n}$. In the case $n=1$ the relation
$xy-yx=y^2$ implies the desired formula with
$\alpha_{1,1}=\alpha_{2,1}=1$. Suppose $n$ is a positive integer and
there exist positive integers $\alpha_{k,n}$, $k=1,\dots,n+1$ such
that $x^ny=\sum\limits_{k=1}^{n+1}\alpha_{k,n}y^kx^{n-k+1}$.
Multiplying the latter equality by $x$ from the left and using lemma
\ref{l3.1.1} we obtain
$$
x^{n+1}y=\sum_{k=1}^{n+1}\alpha_{k,n}xy^kx^{n-k+1}=
\sum_{k=1}^{n+1}\alpha_{k,n}y^kx^{n-k+2}+
\sum_{k=1}^{n+1}\alpha_{k,n}ky^{k+1}x^{n-k+1}.
$$
Rewriting the second term as
$\sum\limits_{k=1}^{n+2}\alpha_{k-1,n}(k-1)y^{k}x^{n-k+2}$ (here
we assume that $\alpha_{0,n}=0$), we arrive to
$$
x^{n+1}y=\sum_{k=1}^{n+2}\alpha_{k,n+1}y^kx^{n-k+2},
$$
where $\alpha_{k,n+1}=\alpha_{k,n}+(k-1)\alpha_{k-1,n}$ for
$k=1,\dots,n+1$ and $\alpha_{n+2,n+1}=(n+1)\alpha_{n+1,n}$.

Let us prove now the formula for $\alpha_{k,n}$. For $n=1$ it is
true since $\alpha_{1,1}=\alpha_{1,2}=1$. Then we use  inductive
argument. Suppose the formula is true for $n$. We are going to apply
the recurrent formula appeared above:
$$ \alpha_{k,n+1}=\alpha_{k,n}+(k-1)\alpha_{k-1,n}=
\frac{n!}{(n-k+1)!}+(k-1)\frac{n!}{(n-k+2)!}=
\frac{(n+1)!}{(n-k+2)!} $$ and the formula is verified for $1\leq
k\leq n+1$.  For $k=n+2$, we have
$\alpha_{n+2,n+1}=(n+1)\alpha_{n+1,n}=(n+1)n!=(n+1)!$. This
completes the proof. $\Box$

\subsection{Koszulity of the Jordan plane}\label{Kosz}

In the section \ref{mult} we have seen that relation $xy-yx-y^2$ forms a Gr\"obner basis with respect to the ordering $x>y$. This means, we can apply to this presentation of the algebra the Priddy criterion of Koszulity \cite{Priddy},
which asserts that if an algebra has a presentation by a quadratic Gr\"obner basis, then it is Koszul.
This leads us to
\begin{proposition}
The Jordan plane $R=k\langle x,y \, | \, xy-yx-y^2 \rangle$ is Koszul.
\end{proposition}

In section \ref{NCCI} this fact will follow also from more general statement related to the properties of noncommutative Koszul (Golod-Shafarevich) complex, but this argument provides a simple constructive and direct proof of Koszulity.

\section{Automorphisms of the Jordan plane}\label{aut}

In this section  we intend to describe  the group of automorphisms
of the Jordan plane $R$ in order to use this later on for
constructing examples of tame up to automorphism strata.  The
automorphism group turns out to be quite small, compared with
automorphisms of the first Weyl algebra $A_1$, which contains $R$ as
a subalgebra. Automorphisms of the $A_1$  were described in
\cite{ML}, the case of an arbitrary Weyl algebra $A_n$ was discussed
in \cite{MK}.

 We  are going to prove.

\begin{theorem}\label{taut}
 All automorphisms of $R=k\langle x,y|xy-yx=y^2\rangle$
are of the form $x\mapsto \alpha x+p(y)$, $y\mapsto \alpha y$, where
$\alpha\in k\setminus\{0\}$ and $p\in k[y]$ is a polynomial on $y$.
Hence the group of automorphisms isomorphic to a semidirect product
of an additive group of polynomials $k[y]$ and a multiplicative
group of the field  $k^*: \,\,{\rm Aut}(R) \simeq k[y]
\leftthreetimes k^*$.
\end{theorem}

\proof Key observation for this proof is that in our algebra there
exists the minimal ideal with commutative quotient. Namely, the
two-sided ideal $J$ generated by $y^2$.

\begin{lemma}\label{lcommq} If the quotient $R/I$ is commutative then $y^2\in I$.  \end{lemma}

\proof The images of $x$ and $y$ in this quotient commute. Hence
$$
I=(x+I)(y+I)-(y+I)(x+I)=xy-yx+I=y^2+I.
$$
Therefore $y^2\in I$. $\Box$

The property of an ideal to be a minimal ideal with commutative
quotient is invariant under automorphisms.

Let us denote by $\widetilde y=f(x,y)$ the image of $y$ under an
automorphism $\phi$. Then the ideal generated by $\widetilde y^2$
coincides with the ideal generated by $y^2$: $J={\rm id} (
y^2)={\rm id}( \widetilde y^2)$.

Using the  property of multiplication in $R$ from lemma
\ref{l3.1.2}, we can see that two-sided ideal generated by $y^2$
coincides with the left ideal generated by $y^2$: $Ry^2R=y^2R$.
Indeed, let us present an arbitrary element of $Ry^2R$ in the form
$\sum a_iy^2b_i$, where $a_i$, $b_i\in R$ are written in the normal
form $a_i=\sum \alpha_{k,l}y^kx^l$, $b_i=\sum \beta_{k,l}y^kx^l$.
Using the relations from Lemma \ref{l3.1.1}, we can pull $y^2$ to
the left through $a_i$'s and get the sum of monomials, which all
contain $y^2$ at the left hand side. Thus,  $\sum a_iy^2b_i=y^2u$,
$u\in R$.

Obviously automorphism maps the one-sided ideal $y^2R$ onto the
one-sided ideal $\widetilde y^2R$, both of which coincide with
$J={\rm id} (y^2)={\rm id}( \widetilde y^2)$. From this we
obtain a presentation of $y^2$ as $\widetilde y^2 u$ for some $u\in
R$. Considering  usual degrees of these polynomials (on the set of
variables $x,y$), we get $2=2k+l$, where $k={\rm deg}\,\widetilde y$
and $l={\rm deg}\,u$. Obviously $k\neq 0$. Hence the only
possibility is $k=1$ and $l=0$.

Thus, $\phi(y)=\widetilde y=\alpha x+\beta y+\gamma$ and $u=c$ for
some $\alpha,\beta,\gamma,c\in k$. Substituting these expressions
into the equality $y^2=\widetilde y^2 u$, we get $c(\alpha x+\beta
y+\gamma)^2=y^2$. Comparing the coefficients of the normal forms
of the right and left hand sides of this equality, we obtain
$\alpha=\gamma=0$, $\beta\neq 0$. Hence $\phi(y)=\beta y$.

Now we intend to use invertibility of $\phi$. Due to it there
exists $\alpha_{ij}\in k$ such that $x=\sum \alpha_{ij}\widetilde
y^i \widetilde x^j$. Substituting $\widetilde y=\beta y$, we get
$\ddd x=\sum_{r=0}^N p_r(y)\widetilde x^r$, where $N$ is a
positive integer, $p_r\in k[y]$ and $p_N\neq 0$. Comparing the
degrees on $x$ of the left and right hand sides of the last
equality we obtain $1=kN$, where $k={\rm deg}_x\widetilde x$.
Hence $k=N=1$, that is $x=p_0(y)+p_1(y)\widetilde x$ and
$\widetilde x=q_0(y)+q_1(y)x$, where $p_0,p_1,q_0,q_1\in k[y]$.
Substituting $\widetilde x=q_0(y)+q_1(y)x$ into
$x=p_0(y)+p_1(y)\widetilde x$, we obtain $q_1\in k$, that is
$\widetilde x=cx+p(y)$ for $c\in k$. One can easily verify that
the relation $\widetilde x\widetilde y-\widetilde y\widetilde x=
\widetilde y^2$ is satisfied for $\widetilde x=cx+p(y)$,
$\widetilde y=\beta y$ if and only if $c=\beta$. This gives us the
general form of the automorphisms:  $\widetilde x=cx+p(y)$,
$\widetilde y=c y, c\neq 0$.

Now we see that the group of automorphisms is a semidirect product
of the normal subgroup isomorphic to the additive group of
polynomials $k[y]$ and the subgroup isomorphic to the multiplicative
group $k^*$. The precisely written formula for multiplication in
${\rm Aut} R$ is the following:

$$\phi_1 \phi_2 = (p_1(y),c_1) (p_2(y), c_2) = (c_2
p_1(y)+p_2(c_1y), c_1c_2)$$

for $\phi_1, \phi_2 \in {\rm Aut} R $. $\Box$

\section{Prime and primitive ideals}\label{prime}

On our way we describe here also prime ideals of the Jordan plane and point out which of them are primitive. It is quite straightforward.

\begin{theorem}\label{aslkidjf} All prime ideals of $R=k\langle
x,y|xy-yx-y^2\rangle$ have a shape ${\rm id}(y)$ or ${\rm
id}(y,x-\alpha)$.
\end{theorem}

\begin{lemma}\label{zlks} Any two-sided proper ideal $I$ in $R$
contains a polynomial in $y$.
\end{lemma}

\begin{proof} For each $f\in R\setminus\{0\}$, let $k(f)$ be the
$x$-degree of the normal form of $f$ and $k(I)=\min\limits_{f\in
I\setminus\{0\}}k(f)$. Assume that $k(I)>0$. Let $f\in I$ be such
that $k(f)=k(I)$. Since $I$ is a two-sided ideal, $g=[f,y]\in I$.
Since $f$ is not a polynomial in $y$, we have $g\neq 0$. On the
other hand, it is easy to see that $k([h,y])<k(h)$ for each $h\in
R$. Hence $k(g)<k(f)$. We have arrived to a contradiction with the
equality $k(f)=k(I)$. Thus $k(I)=0$ and therefore there is $f\in
I\setminus\{0\}$, being a polynomial in $y$.
\end{proof}

We shall take into account that in purely differential Ore
extensions quotients by prime ideals are domains. So we can
substitute primeness with completely primeness in our proof.

\begin{lemma}\label{zlksz} Let $P$ be a prime two-sided ideal of the
Jordan plane $R$. Then $y\in P$.
\end{lemma}

\begin{proof} By Lemma~\ref{zlks} there is a non-zero polynomial $p$
such that $p(y)\in P$. Since $ k$ is algebraically closed, we can
assume that  $p(y)=\prod(y-\lambda_j)$, $\lambda_j\in{k}$. If all
$y-\lambda_j\notin P$, we have zero divisors in $R/P$. Thus
$y-\lambda_j\in P$ for some $\lambda_j\in { k}$. Now we shall show
that $\lambda_j=0$. Indeed, since $y-\lambda_j\in P$, $y=\lambda_j$ in $R/P$.
In the quotient $R/P$ we still have the defining
relation of $R$, so $[x,\lambda_j]=\lambda_j^2$ and $\lambda_j^2=0$,
which implies $\lambda_j=0$.
\end{proof}

\begin{proof}(of Theorem~\ref{aslkidjf}).

Let us notice that the ideal generated by $y$ is prime in $R$.
Indeed, $R/{\rm id}(y)={ k}[x]$ is a domain, which means that
${\rm id}(y)$ is completely prime, and by the above remark is
prime.

Ideal ${\rm id}(y,x-\alpha)$ is also prime due to a similar
reason. Indeed, $R/{\rm id}(y,x-\alpha)={ k}$ is a domain. We
have to demonstrate that there are no other prime ideals. Let $P$ be
a proper prime ideal in $R$. By Lemma~\ref{zlksz}, $y\in P$. Thus,
$P$ can be generated just by $y$ or the ideal $P$ is generated by
$\{y\}\cup{\cal F}$, where $\cal F$ is a subset of ${ k}[x]$. Then
$R/P=(R/{\rm id}( y))/{\rm id}({\cal F})={ k}[x]/{\rm id}(
{\cal F})$. Since ${ k}[x]$ is a principle ideal domain, we
have ${\rm id}({\cal F})={\rm id}(g)$, where $g\in {
k}[x]$. Thus, $R/P={ k}[x]/{\rm id}(g)$. Since $R/P$ must be a
domain, we have $g=x-\alpha$ for some $\alpha\in{ k}$. Thus,
$P={\rm id}( y,x-\alpha)$.
\end{proof}

Now we can easily see which of these ideals are primitive.

\begin{corollary}\label{zmhsdfksj} The complete set of primitive
ideals in $R$ consists of the ideals ${\rm id}(y,x-\alpha)$,
$\alpha\in { k}$.
\end{corollary}

\section{Irreducible modules,
description of all finite dimensional modules}\label{descript}

We describe here all irreducible and completely reducible modules
over the Jordan plane. For arbitrary module, we give a description  of
 $X=\rho(x)$, in terms of ${\cal B}-Toeplitz$ matrices, subject to
 the Jordan normal form of  the matrix $Y=\rho(y)$. This kind of description of
the set of all representations will be useful for obtaining results in the
stratification of representation variety, we suggest later
in the paper.

We will need the following definitions.

\begin{definition} Let us remind that upper triangular (rectangular) Toeplitz
matrix is a matrix with entries $a_{ij}$ defined only by the
difference $i-j$ and  it has zeros below the main diagonal (or upper
main diagonal in a proper rectangular case).
\end{definition}

\begin{definition}
We call a matrix corresponding to the partition {$\cal P$}
block-upper triangular Toeplitz ($\cal{B}$-Toeplitz), if all
blocks, defined by the partition, including diagonal blocks, are
upper triangular (rectangular) Toeplitz matrices.

\end{definition}

\begin{definition}

We call a matrix corresponding to the partition {$\cal P$}
$\cal{J}$-Toeplitz  if it is a sum of $\cal{B}$-Toeplitz matrix and
a matrix with diagonal blocks having the sequence $0,1,2,...$ on the
first upper diagonal and zeros elsewhere.
\end{definition}

\begin{theorem}\label{tdm} The complete set of finite
dimensional representations of $R$ (subject to the Jordan normal form of
$Y$) can be described as a set of pairs of matrices $(X_n,Y_n)$,
where $Y_n$ is in the Jordan form (with zero eigenvalues) corresponding to the partition {$\cal
P$} of $n$ and $X_n$  is a $\cal{J}$-Toeplitz matrix defined by
{$\cal P$}.

\end{theorem}

From this theorem immediately follows a precise description of all
{\it irreducible} and {\it completely reducible} modules.

\begin{corollary}\label{c3.1}
A complete set of pairwise  non-isomorphic finite dimensional {\it
irreducible} $R$-modules is $\{ S_{a} | {a} \in k\}$, where $S_{a}$
defined by the following action of $X$ and $Y$ on one-dimensional
vector space: $Xu=\a u, Yu=0.$

All {\it completely reducible} representations are given by
matrices: $Y_n=(0)$, $X_n$ is a diagonal matrix ${\rm diag}
(a_1,...,a_n)$.
\end{corollary}

\proof Let us describe an arbitrary representation $\rho_n:R\to
M_n(k)$ of $R$, for $n\in\N$. We can assume that the image of one of
the generators $Y=\rho_n(y)$ is in the Jordan normal form.

{\it The full Jordan block case.}

 Let us first find all possible matrices
$X=\rho_n(x)$ in the case when $Y$ is just the full Jordan block:
$Y=J_n$. We have to find then matrices $X=(a_{ij})$ satisfying the
relation $[X,J_n]=J_n^2$. Let $B=[X,J_n]=(b_{ij})$, then
$b_{ij}=a_{i+1,j}-a_{i,j-1}$. From the condition $B=J_n^2$ it
follows that $b_{ij}=0$ if $i\neq j-2$ and $b_{ij}=1$ if $i=j-2$.
Here and later on we will  use the following numeration of
diagonals: main diagonal has number 0, upper diagonals have positive
numbers $1,2,\dots,n-1$ and lower diagonals have negative numbers
$-1,-2,\dots,-n+1$.

The first condition above means that in the matrix $X$ elements
of any diagonal with number $0 \leq k\neq 1$ coincide and are zero
for $k<0$. From the second condition it follows that the elements of
the first upper diagonal form an arithmetic progression with
difference 1:  $a+1,\dots,a+n-1$.

Denote by $X_n^0$ a matrix with the sequence $0,1,2,...$ on the first upper
diagonal and zeros elsewhere.  Then our family of representations
consists of pairs of matrices $(X_n,Y_n)=(X_n^0+T, J_n)$, where $T$
is an arbitrary upper triangular Toeplitz matrix.

{\it The case of an arbitrary partition}.

 Consider now the
general case when the Jordan normal form of $Y$ contains several
Jordan blocks: $Y=(J_1,...,J_m)$, corresponding to the partition $\cal P$.

Cut an arbitrary matrix $X$ into the square and rectangular blocks
of sizes defined by $\cal P$. Denote the blocks by $A_{ij}, \, i,j=\overline
{1,m}$.

Then we can describe the structure of the matrix $B=[X,Y]$ in the
following way:

{\small

$$
B=\left(\Msix\right),\ \ \text{where\ }[A_{ij}]=A_{ij}J_i-
J_jA_{ij}.
$$}
From the condition $B=Y^2$ we have that $[A_{ii},J_i]=J_i^2$ and
hence $A_{ii}$ is the same as in the previous case when $Y$ was just
a full Jordan block and $A_{ij}J_i-J_jA_{ij}=0$ for $i\neq j$. The
latter condition means that $A_{ij}$ for $i\neq j$  are upper rectangular Toeplitz matrices.
Hence $X$ has a shape of ${\cal J}$-Toeplitz matrix.

As a result we have the
family of representations described in Theorem \ref{tdm}. $\Box$

Using arguments analogues to the  above we can ensure.

\begin{proposition}\label{cstr}
Let $Y$ be a matrix in the Jordan normal form defined by the partition {$\cal
P$},  then the centralizer of $Y$, $C(Y)=\{ Z \in Gl_n(k) \,|\,
ZY=YZ \}$ consists of all $\cal{B}$-Toeplitz matrices corresponding
to  {$\cal P$}.

\end{proposition}

\section{$R$ is residually finite dimensional
}\label{rfd}

Let us consider now one of the sequences of representations
constructed in the  previous section: $\ee_n: R \rightarrow {\rm
End}\,\, k^n$, defined by $\ee_n(y)=J_n$, $\ee_n(x)=X_n^0$.
Note that this sequence is basic in the following sense. All representations
corresponding to $Y$ with full Jordan block could be obtained from
$\ee_n$ by the following automorphism of $R$, \, $\phi:R
\longrightarrow R: x \mapsto x+a, y \mapsto y$ where $a \in R$ such
that $[a,y]=0$.

In addition to the conventional equivalence relation on the
representations given by simultaneous conjugation of matrices:
$\rho' \sim \rho''$ if there exists $g \in GL(n)$ such that
$g\rho'g^{-1}=\rho''$ or equivalently, R-modules corresponding to
$\rho'$ and $\rho''$ are isomorphic, we introduce here one more
equivalence relation.

\begin{definition} We say that two representations of the
algebra $R$ are {\it auto-equivalent} (equivalent up to
automorphism) $\p' \sim_A \p'' $ if  there exists $\phi \in \rm{Aut}
(R)$ such that $ \p' \phi \sim \p''$.
\end{definition}

So  we can state that any full block representation
 is auto-equivalent to $\ee_n$ for appropriate $n$.

We will prove now that the sequence of representations $\ee_n$
asymptotically is faithful.

Start with the calculation of matrices which are image of monomials
$y^k x^m$ under the representation $\ee_n$.

\begin{lemma}\label{lImyx} For the representation $\ee$ as above
the matrix $\ee(y^k x^m)$ has the following shape: there is only one
nonzero diagonal, number $k+m$. In the diagonal
appears the sequence $p(0),p(1),...,p(j),...$ of values
of polynomial $p(j)=(k+j)...(k+m+j-1)=\prod_{i=1}^m (k+j+i)$ of degree
$m$.
\end{lemma}

\proof Image $\ee(x^m)$ of the monomial $x^m$ is a matrix with
vector $[1 \cdot 2 \cdot...\cdot m, 2 \cdot 3 \cdot ...\cdot (m+1),
...] $ on the (upper) diagonal number $m$ in the above numeration
and zeros elsewhere. Multiplication by  $\ee (y^k)$ acts on matrix
by  moving up all rows on $k$ steps. We can  now see that matrix
corresponding to the polynomial $y^k x^m$ can have only one nonzero
diagonal, number $m+k$, and vector in this diagonal is the
following: $[(k+1)...(m+k), (k+2)...(m+k+1),...]$.$\Box$

\begin{theorem}\label{t4.1} Let $\ee_n$ be the sequence of representations of $R$ as above.
Then $\cap_{n=0}^\infty \, {\rm ker} \, \ee_n = 0$. \end{theorem}

\proof We are going to show that  $\ee_n(f) \neq 0$ for $n \geq 2
\deg f$. Suppose that for sufficiently large $n$,   $\ee_n(f)$ is
zero and get a contradiction. Denote by $l$ the  degree of polynomial
$f$, and let $f=f_1 +...+f_l$ be a decomposition of $f \in R$ on the
homogeneous components of degrees $i=1,...,l$ respectively. From
Lemma \ref{lImyx} we know  the shape of the matrix which is an image of a
monomial $y^k x^m$.

Applying Lemma \ref{lImyx} to each homogeneous part of the given
polynomial $f$ we get that
$$
f_l=\sum_{k+m=l}a_{k,m}y^kx^m=\sum_{r=0}^la_ry^{l-r}x^r
$$
is a sum of matrices $\displaystyle \sum_{r=0}^la_rM_r$, where $M_r$
has the vector $[(p(0),...,p(j)]$:
$$
\left(\prod_{i=1}^r(l-r+i),\prod_{i=1}^r(l-r+i+1),\dots,
\prod_{i=1}^r(l-r+i+j),\dots \right)
$$
on the diagonal number l (all other entries are zero). The number on
the $j$-th place of this diagonal is the value in $j$ of the
polynomial
$$P(j)=
(l-r+j) \cdot ...\cdot (l+j-1)$$ of degree exactly $r$. Therefore
the sum $\displaystyle \sum_{r=0}^l a_rM_r$ has a polynomial on $j$
of degree $N=\max\{r:a_r\neq0\}$ on the diagonal number $l$. Since
any polynomial of degree $N$ has at most $N$ zeros we arrive to a
contradiction in the case when $l$th diagonal has length more than
$l$. Hence for any $n \geq 2 \deg f$, \, $\ee_n(f) \neq 0$. $\Box$

Let us recall that an algebra $R$ {\it residually has some property
$P$} means that there exists a system of equivalence relations
$\tau_i$ on $R$ with trivial intersection, such that in the quotient
of $R$ by any $\tau_i$ property $P$ holds.

From the Theorem \ref{t4.1}  we have the following corollary
considering equivalence relations modulo ideals ker $\ee_n$.

\begin{corollary}\label{c4.1} Algebra $R$ is residually finite dimensional.\end{corollary}

\section{Indecomposable modules}\label{ind}

\begin{lemma}\label{lam}
Let $M=(X,Y)$ be a (finite dimensional) indecomposable module over
$R=k\langle x,y|xy-yx=y^2\rangle$. Then $X$ has a unique eigenvalue.
\end{lemma}

\proof Denote by $M_\lambda^X$ the main eigenspace for $X$
corresponding to its eigenvalue $\lambda$:
$M_\lambda^X=\bigcup\limits_{k=0}^\infty {\rm ker}\,(X-\lambda
I)^k$. Obviously $M_\lambda^X={\rm ker}\,(X-\lambda I)^m$, where $m$
is the maximal size of blocks in the Jordan normal form of $X$. It
is well-known that $M=\mathop{\oplus}\limits_{i} M_{\lambda_i}^X$,
where the direct sum is taken over all different eigenvalues
$\lambda_i$ of $X$. We shall show that $M_{\lambda_i}^X$ are in fact
$R$-submodules.

Let $u\in M_{\lambda}^X$, that is $(X-\lambda I)^mu=0$. We calculate
$(X-\lambda I)^nYu$ for arbitrary $n$. Using the fact that the
mapping defined on generators $\phi(x)=x-\lambda$, $\phi(y)=y$
extends to an automorphism of $R$ (see \ref{aut}), we can apply it
to the multiplication  formula from Lemma \ref{l3.1.2} to get
$(x-\lambda)^ny=\sum\limits_{k=1}^{n+1}y^k(x-\lambda)^{n-k+1}$.
Taking into account that $Y^l=0$ for some positive integer $l$, we
can choose $N$ big enough, for example  $N\geq m+l$, such that

$$
(X-\lambda I)^NYu=\sum_{k=1}^{N+1}\alpha_{k,N}Y^k(X-\lambda
I)^{N-k+1}u=0
$$

either due to $(X-\lambda I)^{N-k+1}u=0$ or due to $Y^k=0$.

 This shows that $Yu\in M_\lambda^X$, that is
$M_\lambda^X$ is invariant with respect to $Y$. $\Box$

As an immediate corollary we have the following.

\begin{theorem}\label{pdecomp} Any finite dimensional $R$-module $M$ decomposes into the
direct sum of submodules  $M_{\lambda_i}^X$ corresponding to
different eigenvalues $\lambda_i$ of $X$. \end{theorem}

\begin{corollary}\label{cindlocal} Let $M$ be indecomposable module corresponding to the
representation $\p: R \longrightarrow {\rm End}(k^n)$, and $A_n$ is
the image of this representation. Then $A_n$ is a local algebra, e.i.
$A_n / J(A_n) = k$. \end{corollary}

\proof This follows from the above lemma  and the fact that any
image algebra is basic  with semisimple part  isomorphic to the sun
of $r$ copies  of the field $k$: $\oplus_{r} k $, where $r$ is a
number of different eigenvalues of $X$, which was proved in proposition
\ref{t2.3}.$\Box$

Now using the definition of quiver for the image algebra given in
section \ref{quiv}  and Lemma \ref{radrsq} we have a complete
description of quiver equivalence classes of indecomposable modules.

\begin{corollary}\label{cindQuiv} Quiver corresponding to the indecomposable module has one
vertex. The number of loops is one or two, which is a dimension of
the vector space ${\rm Span}_k\{ \bar X-\lambda I, \bar Y\}$, where
$\bar  X=\phi X$, $\bar  Y=\phi Y$ for $\phi: A \rightarrow A/J^2$.
\end{corollary}

As another consequence of  Theorem \ref{pdecomp} we can
derive an important information on how  to glue irreducible modules
to get indecomposables. It turns out that it is possible to glue
together nontrivially  only the copies of the same irreducible
module $S_a$.

\begin{corollary}\label{ext} For arbitrary non-isomorphic irreducible   modules
$S_a, S_b$,  $${\rm Ext}^1_R (S_a,S_b)=0, \,\,{\rm if}\,\, a \neq
b.$$
\end{corollary}
\proof Indeed, in Corollary \ref{c3.1} we derive that irreducible
module $S_i$ is one dimensional and given by $X=(a), Y=(0)$, $a \in
k$. If $a \neq b$ then for $[M] \in {\rm Ext}^1_R (S_a,S_b)$,
corresponding $X$ has two different eigenvalues, namely $a$ and $b$.
Then by the above lemma $M$ is decomposable and $[M]=0$. $\Box$

\section{Equivalence of some subcategories in mod$\,R$}\label{eqc}

Let us denote by mod $R(\lambda)$ the full subcategory in mod$\,R$
consisting of modules with the  unique eigenvalue $\lambda$ of $X$:
${\rm mod}\,R(\lambda)=\{M\in {\rm mod}\,R|M=M_\lambda(X)\}$. Let us
define the functor $F_\lambda$ on ${\rm mod}\,R$, which maps a
module $M$ to the module $M_\lambda$ with the following new action
$rm=\phi_\lambda(r)m$, where $\phi_\lambda$ is an automorphism of
$R$ defined by $\phi_\lambda(x)=x+\lambda$, $\phi_\lambda(y)=y$. The
restriction of $F_\lambda$ to ${\rm mod}\,R(\lambda) $ is an
equivalence of categories $F_\lambda:{\rm mod}\,R(\lambda)\to {\rm
mod}\,R(\mu+\lambda)$ for any $\mu\in k$. In particular, we have an
equivalence of the categories ${\rm mod}\,R(\lambda)$ and ${\rm
mod}\,R(0)$.

To use this equivalence of categories it is necessary to know that
in most cases (but not in all of them), the eigenvalues of the
matrix $X$ are just entries of the main diagonal in the standard
shape of the matrix described in the Theorem \ref{tdm}, more
precisely.

\begin{theorem}\label{tnn} Let in the
basis ${\cal E}$ of the representation vector space, $Y$ is in the
Jordan normal form, and Jordan blocks have pairwise different sizes:
$n_1,n_2,\dots,n_k$. Then in the same basis $X$ is a ${\cal J}$-Toeplitz
matrix corresponding to the partition ${\cal P}=n_1,n_2,\dots,n_k$ with
numbers $\lambda_1,\dots,\lambda_k$ on the
diagonals of the main blocks, where $\lambda_j$ are eigenvalues of
$X$ (not necessarily different).
\end{theorem}

\proof Let us first introduce the denotation for the basis $\cal E$:

$$
e^{1,1},\dots,e^{1,n_1},e^{2,1},\dots,e^{2,n_2},\dots,
e^{k,1},\dots,e^{k,n_k}.
$$

Consider the set $\cal A$ of matrices (in the basis ${\cal E}$) such
that $A_{(j,l),(j,l)}=c_j$, $1\leq j\leq k$, $1\leq l\leq n_j$ and
$A_{(i,s),(j,l)}=0$ if $n_j<n_i$, $l>s-n_j$ and if $n_j>n_i$,
$l>s$. One can easily verify that $\cal A$ is an algebra with
respect to the matrix multiplication. Let also $\cal D$ be the
subalgebra of diagonal matrices in ${\cal A}$ and $\phi:{\cal A}\to
{\cal D}$ be the natural projection ($\phi$ acts by annihilating the
off-diagonal part of a matrix).

Looking at the multiplication in $\cal A$ it is straightforward
that $\phi$ is an algebra morphism, under the condition that $n_1,...,n_k$
are pairwise different numbers, that is $\phi(I)=I$,
$\phi(AB)=\phi(A)\phi(B)$ and $\phi(A+B)=\phi(A)+\phi(B)$. It is
also easy to check, calculating the powers of the matrix, that if
$A\in\cal A$ and $\phi(A)=0$ then the matrix $A$ is nilpotent. Since
${\cal J}$-Toeplitz matrices  belong to $\cal A$, it suffices
to verify that the eigenvalues of any $A\in \cal A$ coincide with
the eigenvalues of $\phi(A)$.

First, suppose that $\lambda$ is not an eigenvalue of $A$. That is
the matrix $A-\lambda I$ is invertible: there exists a matrix
$B\in\cal A$ such that $(A-\lambda I)B=I$. Here we use the fact that
if a matrix from a subalgebra of the matrix algebra is invertible,
then the inverse belongs to the subalgebra. Then $\phi((A-\lambda
I))\phi(B)=\phi((A-\lambda I)B)=\phi(I)=I$. Therefore $\lambda$ is
not an eigenvalue of $\phi(A)$. On the other hand, suppose that
$\lambda$ is not an eigenvalue of $\phi(A)$. Then $\phi(A)-\lambda
I$ is invertible. Clearly
$$
A-\lambda I=(\phi(A)-\lambda I)(I+(\phi(A)-\lambda
I)^{-1}(A-\phi(A))).
$$
Let $B=(\phi(A)-\lambda I)^{-1}(A-\phi(A))$. Since $\phi$ is a
projection, we have that
$$
\phi(B)=\phi((\phi(A)-\lambda I)^{-1})(\phi(A)-\phi(A))=0.
$$
As we have already mentioned this means that the matrix $B$ is
nilpotent and therefore $I+B$ is invertible. Hence $A-\lambda
I=(\phi(A)-\lambda I)(I+B)$ is invertible as a product of two
invertible matrices. Therefore $\lambda$ is not an eigenvalue of
$A$. Thus, eigenvalues of $A$ and $\phi(A)$ coincide. This completes
the proof. $\Box$

\section{Analogue of the Gerstenhaber theorem for commuting
matrices}\label{Ger}

In this section we intend to prove an analog of the
Gerstenhaber-Taussky-Motzkin theorem (see \cite{Ger}, \cite{MT},
\cite{Gur}) on the dimension of images of representations of two
generated algebra of commutative polynomials $k[x,y]$. This theorem
says that any algebra generated by two matrices $A,B\in M_n(k)$ of
size $n$ which commute $AB=BA$ has dimension not exceeding $n$. It
was proved using different means, for example,  in \cite{Gur} one
can find arguments, where irreducibility of commuting variety is
involved and in \cite{W} purely module theoretic methods were used.

Instead of commutativity we consider the relation $XY-YX=Y^2$ and
prove the following

\begin{theorem}\label{tG}
 Let $\rho_n:R\to M_n(k)$ be an arbitrary
$n$-dimensional representation of $R=k \langle x,y|xy-yx=x^2
\rangle$ and $A_n=\rho_n(R)$ be the image algebra. Then the
dimension of $A_n$ does not exceed $\frac{n(n+2)}4$ for even $n$ and
$\frac{(n+1)^2}4$ for odd $n$.

This estimate  is optimal and attained for the image algebra
corresponding to the full block $Y$.

\end{theorem}

We divide the proof in two lemmas. Start with the second statement
of the theorem, that is calculation of the dimension of the image
algebras in the full block case.

Let us note first the following

\begin{lemma}\label{unim}
Let $\rho_n:R \to M_n(k)$ be an $n$ dimensional representation of
$R$, where $Y=\rho(y)$ has a full block Jordan structure. Then
$\rho_n$ is auto-equivalent to the fixed representation $\ee_n, \,
\ee_n(x)=X_n^0, \, \ee_n(y)=J_n$. For any $n$ the image
 algebra $A_n=\rho_n(R)$
does not depend on the choice of $\rho_n$. \end{lemma}

We want to emphasize that the fact we will use here, that $Y$
commutes only with polynomials on $Y$, is specific for the full
block case. We will try to explain why it is so in the course of the
proof.

Let us remind that a matrix $Y \in M_n(k)$ called {\it
non-derogatory} if its characteristic polynomial coincides with the
minimal polynomial, or if any eigenspace has  dimension 1.

It is well-known that
\begin{proposition}\label{nd}
The matrix $Y$ is non-derogatory iff $C(Y)=Alg(Y)$, where $C(Y)$ is
a centralizer of $Y$ in $M_n(k)$ and $Alg(Y)$ --- algebra generated
by $Y$.

\end{proposition}

\proof Consider a representation $\rho \,' \sim \rho$, such that
 $\rho \,'(Y)= J_n$ is a full Jordan block. Denote $Y$=$\rho \,
 (y)=\ee(y)$. Since both pairs $\rho \,'(x) $,  $\rho \,'(y) $ and $
 \ee(x)$, $ \ee(y)$ satisfy the algebra relation, we have $[Y, \rho \,'(x) -\ee(x)]=0. $
 Now, since $Y$ is nilpotent it has one block Jordan structure iff
 it is non-derogatory. This means that only
in case of one block $Y$  we will have that $ \rho \,'(x) -\ee(x)
\in C(Y)$ is a polynomial on $Y$ (due to Proposition \ref{nd}). So,
if indeed, $Y=J_n$, then $ \rho \,'(x) =\ee(x) = p(Y), \, p \in
k[t]$ and $ \rho \,'(x) =\ee(\phi(x)), \,  \rho \,'(y)
-\ee(\phi(y))$, where $\phi(x)=x+p(y), \phi(y)=y$, that is $\phi \in
{\rm Aut} R$. This means that $ \rho \,'$ and $\ee$ are
auto-equivalent, hence also $ \rho$ and $\ee$. From this immediately
follows that they have the same image algebras.
$\Box$

\begin{lemma}\label{tGfull-bl}
 Let $X,Y\in M_n(k)$ be matrices of  size $n$ over the field
$k$, satisfying the relation $XY-YX=Y^2$  and $Y$ has as a Jordan
normal form one full block. Denote by ${A}$ the algebra generated by
$X$ and $Y$. Then for odd $n$, ${\rm dim}\,
A=\frac{(n+1)^2}{4}$ and for even $n$, ${\rm dim}\,
A=\frac{n(n+2)}{4}$.

\end{lemma}

\proof In  Lemma \ref{lImyx} we already have computed the matrices,
which are images of monomials $y^kx^m$ under the representation
$\ee: (x,y)\mapsto (X^0,J_n)$. Due to the Lemma \ref{unim} we have
only one image algebra  for any $n$. In order to calculate its
dimension,
let us recall how matrices $\epsilon(y^kx^m)$ look like and
calculate the dimensions of their linear spans.

 The matrix
$\epsilon(y^{l-r}x^r)$ on the $l$-th upper diagonal has a vector
$(p(0),p(1),\dots)$, where
$$
p(j)=(l-r+j)\dots(l+j-1)=\prod_{i=1}^r(l+j-r+i)
$$
and zeros elsewhere. In the $j$-th place of the $l$-th diagonal we
have a value of a polynomial of degree exactly $r$. Those diagonals
which have number less then  the number of elements in it give the
impact to the dimension equal to the dimension of the space of
polynomials of corresponding degree. When the diagonals become
shorter (the number of elements less then the number of the
diagonal) then the impact to the dimension of this diagonal equals
to the number of the elements in it. Thus, if $n=2m+1$, ${\rm dim}\,
A=1+\dots+m+(m+1)+m+\dots+1=(m+1)^2=\frac{(n+1)^2}{4}$. When $n=2m$,
we have ${\rm dim}\, A=1+\dots+m+m+\dots+1=m(m+1)=\frac{n(n+1)}{4}$.
$\Box$

\subsection{Maximality of the dimension in the full block case}

\def\frac#1#2{{#1}\over{#2}}
\def\text#1{{\rm {#1}}}

We know now that any representation of $R$,
which is isomorphic to one with $Y$ in the full block Jordan normal form
gives us as an image  the same algebra, described in Lemma
\ref{tGfull-bl} as a certain set of matrices, of dimension
${n(n+2)\over 4}$ for even $n$ and $\frac{(n+1)^2}4$ for odd $n$.
We intend to prove that this dimension is maximal among dimensions
of all image algebras for arbitrary representation, that is the
first part of Theorem \ref{tG}.

We start with the proof  that this dimension is an upper bound for
any image algebra of an indecomposable representation.

A simple preliminary fact we will need is the following.

\begin{lemma}\label{ltr}
Matrices  $X$ and $Y$ satisfying the relation $XY-YX=Y^2$ can be by
simultaneous conjugation brought to a triangular form.
\end{lemma}

\proof
Using the defining relation and the fact that $Y$ is nilpotent we
can see that any eigenspace of $Y$ is invariant under $X$. Hence $X$
and $Y$ has joint eigenvector $v$. Then we consider quotient
representation on the space $V/\{v\}$ which has the same property.
Continuation of this process supplies us with the basis where both $X$
and $Y$ are triangular.  $\Box$

\begin{lemma}\label{tGerst}
Let $\rho_n:R\to M_n(k)$ be an indecomposable  $n$-dimensional
representation of $R$ and $A_n=\rho_n(R)$ be the image algebra. Then
the dimension of $A_n$ does not exceed $\frac{n(n+2)}4$ for even $n$
and $\frac{(n+1)^2}4$ for odd $n$. \rm
\end{lemma}

\proof The algebra $A_n=\{\sum \alpha_{k,m}Y^kX^m\}$ consists now of
triangular matrices. Let us present the linear space $UT_n$ of upper
triangular $n\times n$ matrices as the direct sum of two subspaces
$UT_n=L_1\oplus L_2$, where $L_1$ consists of matrices with zeros on
upper diagonals with numbers $l,\dots,n$ and $L_2$ consists of
matrices with zeros on upper diagonals with numbers $1,\dots,l-1$,
where $l=(n+1)/2$ for odd $n$ and $l=n/2+1$ for even $n$. Let $P_j$,
$j=1,2$ be the linear projection in $UT_n$ onto $L_j$ along
$L_{3-j}$. Since $A_n$ is a linear subspace of $UT_n$, we have that
$A_n\subset M_1+M_2$, where $M_j=P_j(A_n)$. Therefore $\text{dim}\,
A_n\leq \text{dim}\,M_1+\text{dim}\,M_2$. The dimension of $M_1$
clearly does not exceed the dimension of the linear span of those
matrices $Y^kX^m$, which do not belong to $L_2$. Thus,
$$
\text{dim}\, M_1\leq \text{dim}\, \langle Y^kX^m| k+m<l-1\rangle _k.
$$

Here we suppose that $X$ (as well as $Y$) is nilpotent. We can do
this because the module is indecomposable. Indeed, lemma
\ref{lam} says that for an indecomposable module $X$ has a unique
eigenvalue. This implies that any indecomposable representation is
autoequivalent to one with nilpotent $X$ and $Y$ due to the
automorphism of $R$ defined by $\phi_{\lambda}(x)=x-\lambda$,
$\phi_{\lambda}(y)=y$. Since autoequivalent representations have the
same image algebras we can suppose that $X$ is nilpotent.

Thus the dimension of $ M_1$ does not exceed the number of the pairs
$(k,m)$ of non-negative integers such that $k+m<l-1$, which is equal
to $1+\dots+(l-1)$. On the other hand $\text{dim}\,M_2\leq
\text{dim}\,L_2$ and the dimension of $L_2$ does not exceed the
total number of entries in the non-zero diagonals.
$$
{\rm dim}\,M_2\leq 1+\dots + (n-l+1).
$$
Taking into account that $\text{dim}\, A_n\leq
\text{dim}\,M_1+\text{dim}\,M_2$, we have
$$
{\rm dim}\, A_n\leq 1+\dots+(l-1)+1+\dots+(n-l+1).
$$
The latter sum equals $\frac{n(n+2)}4$ for even $n$ and
$\frac{(n+1)^2}4$ for odd $n$.
$\Box$

After we have proved the estimation for the indecomposable modules,
it is easy to see that the same estimate holds for arbitrary module,
since the function $n^2$ is convex.

On the other hand as it was shown in the Lemma \ref{tGfull-bl} that this
estimate is attained on the algebra $A_n=\epsilon(R)$ in the case of
the full block $Y$.
This completes the proof of  Theorem \ref{tG}.

\section{Stratification of the Jordan variety}\label{param}

Here we suppose that $k=\C$.
 Let us consider the variety of
$R$-module structures on $k^n$ and denote it by $mod(R,n)$. Such
structures are in 1-1 correspondence to  $k$-algebra homomorphisms
$R\to M_n(k)$ ($n$-dimensional representations), or equivalently to
a pairs of matrices $(X,Y)$, $X,Y\in M_n^{(2)}(k)$, satisfying the
relation $XY-YX=Y^2$. The group $GL_n(k)$ acts on $mod(R,n)$ by
simultaneous conjugation and orbits of this action are exactly the
isomorphism classes of $n$-dimensional $R$-modules. Denote this
orbit of a module $M$ or of a pair of matrices $(X,Y)$ as ${\cal
O}(M)$ or ${\cal O}(X,Y)$ respectively.

Consider the following
{\it stratification} on $mod(R,n)$. Let ${\cal U_P}$ be the set of all pairs $(X,Y)$
satisfying the relation, where $Y$ has a fixed Jordan form. Here
$\cal P$ stands for the partition of $n$, which defines the Jordan
form of $Y$. Clearly ${\cal U_P}$ is a union of all orbits where $Y$
has a Jordan form defined by the partition $\cal P$.
We
will write ${\cal U}_{(n)}$ for the stratum corresponding to the trivial partition
${\cal P}=(n)$ or to $Y$
with the full Jordan block.

In this stratum ${\cal U_P}$ there is a natural choice of the {\it vector} bundle structure with the base consisting of
$Y$s with the given Jordan form defined by $\cal P$. The fiber then will be an affine space of solutions of the equation $XY-YX=YY$ with respect to $X$, when $Y$ is fixed. The dimension does not depend on the $Y$, since $Y$s are all conjugate.  It is equal to the dimension of the space of ${\cal B}-Toeplitz$ matrices defined by the partition $\cal P$. If ${\cal P}=(n_1 \leq n_2... \leq n_r)$, then the dimension of this fiber is $m=(2r-1)n_1+(2r-3)n_2+...+3n_{r-1}+n_r$. The dimension of the base is therefore $n^2-m$.
For example, for the strata ${\cal U}_{(n)}$ defined by the full block $Y$, a generic one, the dimensions of the base and the fiber in this vector bundle distributed as $n^2-n$ and $n$ respectively.

{Remark}. If one consider the whole space of pairs of matrices ${\cal M}_n$, in general position there exists an {\it algebraic} bundle, described recently in the Kontsevich Arbeitstagung talk \cite{Maxim}. There dimensions of the fiber and the base distributed as $n^2-2n-1$ and $(n+1)^2$ respectively, fibers are not a vector spaces, but still quite nice -- they are 'algebraic'.  If we try to construct the bundle in our subvariety,  in generic situation,  on the same principle, where the base consists of polynomials ${\rm det} (1+sX+tY)$, and the fibers of pairs of matrices with this polynomial, we get a bundle with one dimensional base. Indeed, the  pairs of matrices $(X,Y)$ with the full block Jordan form of $Y$ are in general position in our variety $mod(R,n)$. Hence in this generic strata $X$  has one eigenvalue and the polynomial, defining the point of the base is ${\rm det}(I-aIs+0It)=(1+as)^n$, so the base is a one parameter family.
The fiber however, is quite complicated: $F_a=\{(X,Y)| X $ with one eigenvalue, $Y$ satisfying $XY-YX=YY \}$. Note that fibers are shifts of each other: $F_b=(X+(b-a)I, Y)=F_a+((b-a)I,0)$.

\subsection{Examples of parametrizable strata}

In this section we will give a parametrization (by two parameters)
of the stratum  ${\cal U}_{(n)}$.

Another action involved here is an action of the subgroup of $GL_n$
on those pairs $(X,Y)$, where $Y=J_{\cal P}$ is in fixed Jordan
form. Denote this space by $W_{\cal P}$. The subgroup which acts
there is clearly the centralizer of the given Jordan matrix:
$C(J_{\cal P})$. Orbits of the action of $C(J_{\cal P})$ on the
space $W_{\cal P}$ are just restrictions of orbits above: ${\cal
O_P}(X)={\cal O}(X,Y) \cap W_{\cal P}$.

We sometimes  consider instead of action of $GL_n$ on the whole
space an action of the centralizer $C(J_{\cal P})$ on the smaller space
$W_{\cal P}$. While the group is not reductive any more
and has a big unipotent part, we act just on the space of matrices
and some information easier to get in this setting. It then could be
(partially) lifted because  of 1-1 correspondence of orbits. More
precisely, it could be lifted if we are interested  in parametrization, but if we
consider, for example,  degeneration of orbits the situation may certainly change
after their  restriction.

What we actually do here is obtaining the parametrization for
$W_{(n)}$. Due to 1-1 correspondence between the orbits we then have
a parametrization of  ${\cal U}_{(n)}$ .

Let us restrict the orbits even a little further, considering the
action of the group $G = C(J_{\cal P}) \cap SL_n$, where the 1-1
correspondence with the initial orbits will be clearly preserved. In
the case ${\cal P}= (n)$ the group $G$ can be presented as follows:
$$G=\{I+\a_1 Y + \a_2 Y^2 + ... + \a_{n-1} Y^{n-1}\},$$
\noindent due to our description of the centralizer of $Y$ in
Proposition \ref{cstr}. This group acts on the affine space of the
dimension $n$:
$$W_{(n)}=\{\l I + X^0+c_1 Y + c_2 Y^2 + ... + c_{n-1} Y^{n-1}\}$$
\noindent here $\l$ is the eigenvalue of $X$ and $X^0$ is the matrix
 with the second diagonal $[0,1,
\dots,(n-1)]$ and zeros elsewhere (defined in section \ref{descript}).


Let us fix first the eigenvalue: $\, \l=0$, we get then the space of
dimension $n-1$:
$$W_{(n)}'=\{ X^0+c_1 Y + c_2 Y^2 + ... + c_{n-1}
Y^{n-1}\}.$$

We intend to calculate now the dimension of the orbit ${\cal
O}_{(n)}(X,G)$ of $X$ with fixed eigenvalue $\l=0$ under $G$ --
action.

Consider the map $\phi: G \longrightarrow W_{Y}'$ defined by this
action: $\phi (C) = CXC^{-1}$, then ${\rm Im}\, \phi ={ \cal
O}_{(n)}(X,G)$. We are going to  calculate the rank of Jacobian of
this map. We will see that it is constant on $G$ and equals to
$n-2$. This tells us that each orbit ${ \cal O}_{(n)}(X,G)$ is an
$n-2$ dimensional manifold and hence there couldn't be more then 2
parameters involved in parametrization of orbits.

\subsection{Calculation of the rank of Jacobian}

\begin{theorem}\label{t7.1} Let $G$ be an intersection of $SL_n$
with the centralizer of $Y$. Consider the action of this group on
the affine space $W_{Y}'=\{ X^0+c_1 Y + c_2 Y^2 + ... + c_{n-1}
Y^{n-1}\}$ by conjugation. Then the rank of the Jacobian of the map
$\phi: G \longrightarrow W_{Y}'$ is equal to $n-2$ at any point $C
\in G$.
\end{theorem}

\proof
Clearly $d \phi(C)(\D)$ is the linear part of the map $(C+\D)^{-1}X(C+\D)-C^{-1}XC,$ where

$$C=I+\a_1 Y + \a_2 Y^2 + ... + \a_{n-1}
Y^{n-1},$$

$$X=X^0+c_1 Y + c_2 Y^2 + ... + c_{n-1} Y^{n-1},$$

$$\D=\b_1 Y + \b_2 Y^2 + ... + \b_{n-1} Y^{n-1}.$$

Let us present $(C+\D)^{-1}$ in the following way:

$$(C+\D)^{-1}=(I+\D C^{-1})^{-1} C^{-1}=$$

$$(I-\D C^{-1} +  {\rm lower \, order \, terms \, on \,}
\D) C^{-1}.$$

Then

$$(C+\D)^{-1}X(C+\D)-C^{-1} X C =  $$

$$(I-\D C^{-1} +
{\rm lower \, order \, terms \, on \,}  \D) C^{-1} X (C+\D) - C^{-1} X
C=$$

$$ - \D C^{-2} X C + C^{-1} X \D + {\rm lower \, order \, terms \, on \,}\D =$$

$$(-\D C^{-1} \cdot C^{-1} X + C^{-1} X \cdot \D C^{-1}) C
+ {\rm lower \, order \,terms \, on \,}\D .$$

Denote  $\tilde \D := \D C^{-1}$ and $\tilde X :=  C^{-1} X$.
Obviously multiplication by $C$ preserves the rank and rank of
linear map $d \phi(C)(\D)$ is equal to the rank of the map $T(\tilde
\D) = [\tilde X, \tilde \D]$.

Here again
$\tilde \D$ has
a form
$$\tilde \D=
\g_1 Y + \g_2 Y^2 + ... + \g_{n-1} Y^{n_1}.$$

Let us compute commutator of $\tilde X$ with $Y^k$. Taking into
account that $C^{-1}$ is a polynomial on $Y$, hence commute with
$Y^k$ and also the relation in  algebra $R$: $ XY^k-Y^kX=k Y^{k+1}.$
We get $ \tilde X Y - Y \tilde X=
C^{-1}XY^k-Y^kC^{-1}X=C^{-1}(XY^k-Y^kX)= C^{-1} k Y^{k+1} $. Hence
$$ \tilde X p(Y) - p(Y) \tilde X = C^{-1} Y^2 p'(Y)$$
\noindent for arbitrary polynomial $p$. Applying this for the
polynomial $\tilde \D$ we get
$$ T(\tilde \D)= [\tilde X, \tilde \D] =
\sum_{k=1}^{n-2} \g_k k C^{-1} Y^{k+1},$$ \noindent hence this
linear map has rank $n-2$. $\Box$

From Theorem \ref{t7.1} we could deduce the statement concerning
parametrization of isoclasses of modules in the stratum ${\cal U}_{(n)}$.

We mean by {\it parametrization} (by $m$ parameters) the existence
of $m$ smooth algebraically independent functions which are constant
on the orbits and separate them.

\begin{corollary}\label{c7.2}
Let ${\cal U}_{(n)}$ be the stratum as above. Then the set of
isomorphism classes of indecomposable modules from ${\cal U}_{(n)}$
could be parameterized by at most two parameters.
\end{corollary}

\proof Directly from Theorem \ref{t7.1} applying the theorem on
locally flat map \cite{DNF} to $\phi: G \longrightarrow W_{(n)}'$ we
have that ${\rm Im} \phi ={\cal O}_{(n)}(X,G)$ is an $n-2$
dimensional manifold. We have to mention here that this is due to
the fact that the image has no self-intersections. This is the case
since the preimage of any point $P$ is connected (it is formed just
by the solutions of the equation $CX=PC$ for $C \in G$). Hence we
can parametrize these orbits lying in the space $W_{(n)}$ of
dimension $n$ by at most two parameters. Due to 1-1 correspondence
with the whole orbits ${\cal O}(X,Y)$ the latter have the same
property.$\Box$

\begin{proposition}\label{p7.1} Parameters $\mu$ and $\lambda$ are invariant under the
action of $G$ on the set of matrices $\displaystyle \tiny
\left\{\left(\begin{array}{cccccc}
\lambda&\mu+1&&&&\\
&\lambda&\mu+2&&\smash{\hbox{\normalsize*}}&\\
&&\lambda&\mu+3&&\\
&&&\ddots&\ddots&\\
&\smash{\hbox{\normalsize0}}&&&\lambda&\mu+n-1\\
&&&&&\lambda\end{array}\right)\right\}.$
\end{proposition}

\proof Direct calculation of $ZMZ^{-1}$ for $Z\in G$ as described
above shows that elements in first two diagonals of $M$ will be
preserved. $\Box$

Hence from Corollary \ref{c7.2} and Proposition \ref{p7.1} we
have the following classification result for
representations with the full Jordan block $Y$.

\begin{theorem}\label{t7.2} Let $P_{\lambda, \mu}$ denotes the pair
$(X_{\lambda, \mu},Y)$, where
$$\tiny X_{\lambda,\mu}=\left(\begin{array}{cccccc}
\lambda&\mu+1&&&&\\
&\lambda&\mu+2&&\smash{\hbox{\normalsize0}}&\\
&&\lambda&\mu+3&&\\
&&&\ddots&\ddots&\\
&\smash{\hbox{\normalsize0}}&&&\lambda&\mu+n-1\\
&&&&&\lambda\end{array}\right), \ \ \
Y=\left(\begin{array}{cccccc}
0&1&&&&\\
&0&1&&\smash{\hbox{\normalsize0}}&\\
&&0&1&&\\
&&&\ddots&\ddots&\\
&\smash{\hbox{\normalsize0}}&&&0&1\\
&&&&&0\end{array}\right).
$$
Every pair $(X,Y) \in  {\cal U}_{(n)} $ is conjugate to $P_{\lambda, \mu}$
for some $\lambda, \mu$. No two pairs $P_{\lambda, \mu}$ with
different $(\lambda, \mu) $ are conjugate.

\end{theorem}

Let us mention that number of parameters does not depends of $n$ in
this case.

\subsection{Examples of tame strata (up to auto-equivalence)}

We give here some examples of tame strata in the suggested
above stratification related to the Jordan normal form of $Y$. We
show, for example,
that the stratum ${\cal U}_{(n-1,1)}$ corresponding to the partition
${\cal P}=(n-1,1)$
 is tame (but not of finite type) with
respect to auto-equivalence relation on modules. The latter was defined in
section \ref{rfd} and its meaning is in gluing together orbits which could be
obtained one from another using automorphism of the initial algebra.
These examples are quite rare, most strata are wild, as the algebra itself.

\begin{theorem}\label{t7.4}
The stratum  ${\cal U}_{(n)}$
has a finite representation type with respect to
auto-equivalence relation on modules.

The stratum  ${\cal U}_{\cal P}$, for
${\cal P}=(n-1,1)$ is tame, that is
parametrizable by one parameter, with respect to auto-equivalence
relation. \end{theorem}

\proof The proof is analogous to the proof of Theorem \ref{t7.2}.
We present here pictures showing how $X$ and $Y$ act on the basis and
where parameters appear:
$$
\begin{array}{ll}
{\cal P}=(n)&\\
&\bull{e_1}\doubarr{y}{x}\bull{e_2}\doubarr{y}{(x,2)}\bull{e_3} \
\ {\dots} \  \ \bull{e_{n-1}}\doubarr{y}{(x,n-1)}\bull{e_n}\\
\phantom0&\\
\phantom0&\\
\phantom0&\\
{\cal P}=(n-1,1)&\\
&\bull{e_1}\doubarr{y}{x}\bull{e_2}\doubarr{y}{(x,2)}\bull{e_3} \ \
{\dots} \  \ \bull{e_{n-2}}\doubarr{y}{(x,n-2)}\bull{e_{n-1}}
\siglearr{(x,\alpha)}\bull{e_{n}}\\ &\uparrow\hskip7.25cm|\\
&\,\hskip.47mm\smash{\raise 9.5pt\hbox to 7.47cm{\hrulefill}}\\
&\hfill\smash{\raise 10pt\hbox{$\scriptstyle
(x,\alpha^{-1})$}}\hfill
\end{array}
$$$\Box$

\section{The full block stratum and Ringel's classification of
complete local algebras}\label{ring}

Above results suggest to study more carefully the stratum corresponding to the full block $Y$,
since due to Corollary \ref{indec0} they are building blocks in the analogue of the Krull-Remark-Schmidt decomposition theorem on the level of irreducible components.
Moreover results on the analogue of Gerstenhaber theorem confirm the exceptional role of this stratum,
as the dimension of an image algebra reach its maximum in this stratum.
Here we are going to find out when image algebras for this stratum are tame and when they are wild.

We shall show the place of image algebras for representations from $U_{(n)}, \, n\in {\mathbb Z}$ in Ringel's classification
of complete local algebras \cite{Ringloc} by calculating their
defining relations. So the main problem we solve here is the following. Given a finitely presented algebra and its finite dimensional  representation. We need to find all defining relations of the image algebra.

As a result we prove  that for representations of small dimensions, for $n \leq 4$,
all image algebras  are {\it tame} and for $n \geq 5$ they are {\it wild} (in the classical sense accepted in the representation theory of finite dimensional algebras).

Recall that a $k$-algebra $A$ is called {\it local} if $A=k\oplus
{\rm Jac}(A)$, where ${\rm Jac}(A)$ is the Jacobson radical of $A$.
One can also consider the {\it completion} of $A$:
$\overline{A}=\lim\limits_{\displaystyle\longleftarrow}A/({\rm
Jac}(A))^n$. An algebra $A$ is called complete if $A=\overline A$.

We have seen in
Corollary \ref{cindlocal}, that any indecomposable representation has a
local algebra as an image, in particular, representations with
full block $Y$ do.
Note also that all image algebras are complete,
since ${\rm Jac}(A)^N=0$ for $N$ large enough. Indeed, we can use
here the Corollary \ref{radft} which describe the radical, or
observe directly that since $A=k\oplus {\rm Jac}(A)$ and $A$
consists of polynomials on $X^0$ and $J_n$ (as an image of one of
representations $\epsilon_n: (x,y) \mapsto (X^0,J_n)$), then $Jac(A)$
consists of those polynomials which have no constant term. The
matrices $J_n$ and $X^0$ are nilpotent of degree $n$ and $n-1$
respectively, hence ${\rm Jac}(A)^{2n}=0$.

Remind also that due to Lemma \ref{unim} for any $n$ we have only one image algebra
for all representations from ${\cal U}_{(n)}$.

\begin{theorem}\label{an5}  The image algebra $A_n$ of a representation
$\rho_n \in {\cal U}_{(n)} $ is wild for any $n\geq 5$. It has a
quotient isomorphic to the wild algebra given by relations
$y^2,yx-xy,x^2y,x^3$ from the Ringel's list of minimal wild local
complete algebras. The image algebras $A_1,A_2$ and $A_3$ are tame.
\end{theorem}

\proof We intend to show that for $n$ big enough, the algebra $A_n$
has a quotient isomorphic to the algebra $W=\langle
x,y|y^2=yx-xy=x^2y=x^3=0 \rangle$, which is number c) in the
Ringel's list of minimal wild local complete algebras
\cite{Ringloc}.

The algebra $W$ is 5-dimensional. Let us consider the ideal $J$ in
$A_n$ generated by the relations above on the image matrices $X$ and
$Y$. This ideal has obviously codimension not exceeding 5. We intend
to show that $J$ has codimension exactly 5 and therefore $W$ should
be isomorphic to $A/J$.

Let us look at the ideal $J$, which is generated by
$\{Y^2,X^2Y,X^3,XY-YX\}$. First, since $XY-YX=Y^2$, $J$ is generated
by $\{Y^2,X^2Y,X^3\}$. It is easy to see that $Y^2$ has zeros on
first two diagonals and the vector ${\bf {1}}=(1,\dots,1)$ on the
third one,  $X^2Y$, $X^3$ have zeros on the first three diagonals.
An arbitrary  element of $A_n$ has the constant vector $c{\bf 1}$
for some $c \in k$ on the main diagonal.
Hence we see that a general element of the ideal $J$ has zeros on
the first two diagonals and the constant sequence $c{\bf 1}$ on the
third one. Taking into account that $A_n$ comprises the upper
triangular matrices that have values of a polynomial of degree at
most $m$ on $m$-th diagonal (Lemma \ref{lImyx}), we see that the
main diagonal gives an impact of $1$ to the codimension of $J$, the
first diagonal gives an impact of $2$ to the codimension of $J$ if
the length of this diagonal is at least $2$ (that is $n\geq 3$) and
the second diagonal, ---  an impact of $2$, if the length of this
diagonal is at least $3$ (that is $n\geq 5$). Thus, dim$\,A/J\geq 5$
if $n\geq 5$. This completes the proof in the case $n\geq 5$.

Tameness of $A_1$ and $A_2$ is obvious. For $A_3$ the statement
follows from the dimension reason: dim$\,A_3=4$, it is less then the
dimensions of all 2-generated algebras from the Ringel's list of
minimal  wild algebras. Since his theorem (theorem 1.4 in
\cite{Ringloc}) states that any local complete algebra is either
tame of has a quotient from the list, $A_3$ can not be wild by
dimension reasons. Hence $A_3$ is tame. $\Box$

Let us consider now the case $n=4$.

\begin{theorem}\label{an4}  Let $\rho_4 \in {\cal U}_{(4)}$ be a four dimensional
representation of the algebra $R$. Then the image algebra
$A_4=\rho_4 (R)$ is given by the relations $k\langle x,y|x^2=-2xy,
xy=yx+y^2,x^3=0\rangle$ and is tame.
\end{theorem}

\proof We intend to show that no one of the algebras from the
Ringel's list of minimal wild algebras can be obtained as a quotient
of $A_4$. After that using the Ringel's theorem, we will be able to
conclude that it is tame. Suppose that there exists an ideal $I$ of
$A_4$ such that $A_4/I$ is isomorphic to $W_j$ for some $j=1,2,3,4$,
where
\begin{align*}
W_1&=k\langle u,v|u^2,uv-\mu vu\ (\mu\neq0),v^2u,v^3\rangle,
\\
W_2&=k\langle u,v|u^2,uv,v^2u,v^3\rangle,
\\
W_3&=k\langle u,v|u^2,vu,uv^2,v^3\rangle,
\\
W_4&=k\langle u,v|u^2-v^2,vu\rangle.
\end{align*}

Since all $W_j$ are 5-dimensional and $A_4$ is 6-dimensional, the
ideal $I$ should be one-dimensional. Due to our knowledge on the
matrix structure of the algebra $A_4$, we can see that there is only
one one-dimensional ideal $I_4$ in $A_4$ and that $I_4$ consists of
the matrices with at most one non-zero entry being in the upper
right corner of the matrix:
$$
I_4=\left\{\left(
\begin{array}{cccc}
0&0&0&c\\0&0&0&0\\0&0&0&0\\0&0&0&0
\end{array}
\right)\right\}.
$$
After factorization by this ideal we get a $5$-dimensional algebra
given by relations
$$
\overline{A}_4=A_4/I=k\langle x,y|x^2=-2xy,
xy=yx+y^2,x^3=0,y^3=0\rangle.
$$
The question now is whether this algebra is isomorphic to one of the
algebras from the above list. Suppose that there exists an
isomorphism $\phi_j:W_j\to \overline{A}_4$ for some
$j\in\{1,2,3,4\}$. Denote $\phi_j(u)=f_j$ and $\phi_j(v)=g_j$.

First, let us mention that $f_j$ and $g_j$ have zero free terms:
$f_j^{(0)}=g_j^{(0)}=0$ because the equalities $\phi_j(u^2)=f_j^2=0$
and $\phi_j(v^3)=g_j^3$ imply $(f_j^{(0)})^2=(g_j^{(0)})^3=0$ and
therefore $f_j^{(0)}=g_j^{(0)}=0$ if $j=1,2,3$ and the equalities
$\phi_4(u^2-v^2)=f_4^2-g_4^2=0$ and $\phi_4(uv)=f_4g_4=0$ imply
$(f_4^{(0)})^2=(g_4^{(0)})^2$ and $f_4^{(0)}g_4^{(0)}=0$ and
therefore $f_4^{(0)}=g_4^{(0)}=0$.

The second observation is that the terms of degree 3 and more are
zero in $A_4$. Therefore we can present the polynomials $f_j$ and
$g_j$ as the sum of their linear and quadratic (on $x$ and $y$)
parts. So, let $f_j=f_j^{(1)}+f_j^{(2)}$ and
$g_j=g_j^{(1)}+g_j^{(2)}$, where
$$
f_j^{(1)}=ax+by,\ \ g_j^{(1)}=\alpha x+\beta y,\ \
f_j^{(2)}=cyx+dy^2,\ \ g_j^{(2)}=\gamma yx+\delta y^2.
$$
In order to get entire  linear part of the algebra $A_4$ in the
range of $\phi_j$ we need to have
\begin{equation}
{\rm det}\,\left|
\begin{array}{cc}
a&b\\\alpha&\beta
\end{array}
\right|\neq0. \label{det}
\end{equation}

For any $j=1,2,3,4$ we are going to obtain a contradiction of the
last condition with the equations on $a,b,\alpha,\beta$ coming from
the relations of the algebra $W_j$.

For instance, consider the case $j=2$. From
$0=u^2=f_j^2=(f_j^{(1)})^2=2(ab-a^2)yx+(b^2+ab-2a^2)y^2$ we get
$2a(b-a)=0$ and $b^2+ab-2a^2=0$. The first equation gives us that
either $a=0$ or $a=b$. In the case $a=0$ the second equation implies
$b=0$ and the equality $a=b=0$ already contradicts (\ref{det}).
Another solution is $a=b\neq 0$. From
$0=uv=f_jg_j=f^{(1)}g^{(1)}=(ax+by)(\alpha x+\beta y)$, substituting
$a=b$, we get $0=a(x+y)(\alpha x+\beta y)=a(\beta-\alpha)(yx+2y^2)$.
Hence $\beta=\alpha$, which together with the equality $a=b$
contradicts (\ref{det}).

In the other three cases one can  get a contradiction with
(\ref{det}) along the same lines, which completes the proof. $\Box$

Combining theorems \ref{an5} and \ref{an4} we get
\begin{theorem}\label{an}  The image algebra $A_n$ of a representation
$\rho_n \in {\cal U}_{(n)} $ is wild for any $n\geq 5$ and tame for $n\leq 4$.
\end{theorem}

\section{Irreducible components of the representation space}\label{IC}

We are still considering the variety $mod(R,n)$ of $R$-module structures
on $k^n$, with $k=\mathbb C$, as it was set in previous section.

In this section we show (in the basis of obtained above facts)
that the Jordan plane can serve as an example of an algebra for
which all irreducible components of the representation variety
$mod(R,n)$ could be described for any $n$. The importance of such
examples was emphasized in \cite{CBSch} and it is mentioned there that
known cases are restricted by algebras of finite representation type
(i.e., there are only finitely many isomorphism classes of
indecomposable $R$ modules) and one example of infinite representation
type in \cite{GPS}. There is some similarity between the algebra generated
by the pair of nilpotent matrices annihilating each other considered  in \cite{GPS} and the Jordan plane, but
while in the case of \cite{GPS} variables $x$ and $y$ act 'independently', there is
much more interaction in the case of the Jordan plane.

In previous section
we suggested stratification of the representation space $mod(R,n)$ related to
the  Jordan normal form of $Y$ and denote by
${\cal U}_{\cal P}$  the strata  corresponding to the partition ${\cal P}$ of $n$.
These strata are going to play a key role in the description of irreducible components.

We are going to show now that each irreducible component of the variety $mod(R,n)$ contains only one stratum and is the closure of this stratum.
The number of irreducible components is therefore  equal to the number of partitions of $n$.

\begin{theorem}\label{m}
Any irreducible component $K_j$ of the representation variety $mod(R,n)$ of Jordan plane contains only
one stratum $U_{\cal P}$ from the stratification related to the Jordan normal form of $Y$, and is the closure of this stratum.

The number of irreducible
components of the variety $mod(R,n)$ is equal to the number of the
partitions of $n$.

\end{theorem}

\begin{proof}
One can see that ${\cal U}_{\cal P}$ for any
${\cal P}$ is a connected analytic submanifold  inside the variety
$mod(R,n)$. Now we will use the fact that if a connected analytic
manifold $U \subset V$ is contained in the union of varieties $V_i$
($V= \cup V_i$), then it should be contained in one of them: $U
\subset V_j$ (see for example \cite{RB}, ch 1).
 Applying this to the decomposition of $V= mod(R,n)$ into an
irreducible components we get that each irreducible component
contains the whole stratum whenever the stratum touches the irreducible
component.

Now we shall show that the stratum $\U_\pp$ can not be contained in
an intersection of two different irreducible components $K_i$ and
$K_j$. In order to do this, let us calculate the dimension of
$\U_\pp$ (as a manifold). It turns out that it does not depend on
$\pp$ and always equals to the dimensions of $K_i$ and $K_j$.

\begin{lemma}\label{dimstr} For any partition $\pp$ of $n$, the
stratum $\U_\pp$ has  dimension $n^2$.
\end{lemma}

\begin{proof}
Note that the manifold $\U_\pp$ carries the natural vector bundle structure
with the base ${\cal B}_\pp$ being the space of matrices $Y$ with the Jordan normal form
$J_\pp$ and the fiber ${\cal F}_Y=\{X: \, XY-YX=Y^2\}.$
Dimension of the base is equal to  $n^2-\dim C(J_\pp)$, where $C(J_\pp)$ is a centralizer of $J_\pp$.
The fiber is a shift of the space of matrices commuting with $J_\pp$, so its dimension is equal to the dimension of $C(J_\pp)$.

Thus we obtain
$$
\dim \U_\pp=\dim C(J_\pp)+(n^2-\dim C(J_\pp))=n^2
$$
for any partition $\pp$.
$\Box$
\end{proof}

Suppose now that $\U_\pp\subset K_i\cap K_j$. Then
$$
n^2=\dim\U_\pp\leq \dim K_i\cap K_j\leq \dim K_j\leq n^2=\dim U
$$
for $j\in\{1,2\}$.  Hence $\dim K_i\cap K_j=\dim K_i=\dim K_j$,
which contradicts irreducibility of $K_j$.

Hence we have a picture,
where all our strata sitting inside irreducible components and are separated
by them.

This ensures at the same time, that the number of
irreducible components does not exceed the number of the partitions.

Due to the above, for any partition $n=n_1+{\dots}+n_m$, we can take the irreducible
component $C_j$ of the variety $mod(R,n_j)$, containing the
strata $\U_{(n_j)}$. Now from those $C_j$ for $1\leq j\leq m$ using
Theorem~1.2 in \cite{CBSch} we are going to construct an irreducible
component $K=\overline{C_1\oplus{\dots}\oplus C_m}$ of the variety
$ mod(R,n)$. Here $\overline{C_1\oplus{\dots}\oplus C_m}$ has
the following meaning. All the modules $M_1\oplus{\dots}\oplus M_m$,
$M_i\in C_i$ together with all the elements of their
$GL_n(\C)$-orbits gives us $C_1\oplus{\dots}\oplus C_m$ and the bar
stands for the Zariski closure. This gives us an irreducible
component provided $\ext(C_i,C_j)=0$ for $i\neq j$, where
$\ext(C_i,C_j)=\min\{\dim {\rm Ext}^1(M,N):M\in C_i,\ N\in C_j\}$. We can
verify the equality $\ext(C_i,C_j)=0$ in our case by taking $M=(X_1, Y_1)\in
C_i$ and $N=(X_2, Y_2)\in C_j$ in such a way that matrices
$X_1$  and $X_2$ have different eigenvalues. Then by
Lemma \ref{lam} an extension $E$ of $M$ by $N$ should be
decomposable, since it is proved  there that indecomposable modules
always have only one eigenvalue for $X$. Surely decomposition will
be on $M$ and $N$, since we can find a basis, where $X$ splits into
main eigenspaces corresponding to its two different eigenvalues. But
as Theorem \ref{pdecomp} says, these main
eigenspaces are also invariant for $Y$ and we get the splitting
$E=M\oplus N$. Thus ${\rm Ext}^1(M,N)=0$ and therefore $\ext(C_i,C_j)=0$
for $i\neq j$ in our situation and we can apply Theorem~1.2 from
\cite{CBSch}.

But the irreducible component
$C_\pp=\overline{C_1\oplus{\dots}\oplus C_m}$, ($C_j$ is an irreducible
component of the variety $mod(R,n_j)$, containing the
strata $\U_{(n_j)}$),  we have just got,
uniquely determines its summands $C_1,\dots,C_m$ according to
Theorem~1.1 from \cite{CBSch}. Thus the partition
$n=n_1+{\dots}+n_m$ corresponding to their dimensions is also
uniquely determined by $C_\pp$. Hence this component contains the strata
$\U_\pp$ with $\pp=(n_1,\dots,n_m)$ of $mod(R,n)$ and there
are at least as many different irreducible components of $
mod(R,n)$ as the number of partitions of $n$.

Taking into account also the above fact that strata are separated by
irreducible components, we
see that any irreducible component has inside exactly one strata and
is its Zariski closure. This completes the description of
irreducible components of $mod(R,n)$.
$\Box$
\end{proof}

After we have this theorem, we can denote irreducible components by
$K_\pp$, for any partition $\pp$ of $n$, where $K_\pp$ is the
closure of the corresponding strata $\U_\pp$.

We also can answer the question in which irreducible components
module in general position is indecomposable.

\begin{corollary}\label{indec0}
Only the irreducible component $K_{(n)}=\overline{\U_{(n)}}$ which
is the closure of the stratum corresponding to the trivial partition
of $n$ (the full block $Y$) contains an open dense subset consisting of
indecomposable modules.
\end{corollary}

\begin{proof} Suppose our irreducible component $K$ is the closure of  a
stratum related to a non-trivial partition $n=n_1+{\dots}+n_m$,
$m>1$. Then there are modules in this component $M=(X,Y)$, where $X$
has different eigenvalues. Hence the Zariski closed set of modules
$M=(X,Y)$ in our component for which $X$ has only one eigenvalue, has
non-empty complement. The latter complement $\Omega$ is a non-empty
Zariski open subset of the irreducible variety $K$ and therefore
$\Omega$ is dense in $K$. Since by Lemma~\ref{lam} the modules
that correspond to $X$ having different eigenvalues are all
decomposable, we see that $M\in K$ in general position is
decomposable. $\Box$
\end{proof}

\section{NCCI and RCI }\label{NCCI}

In this section we are going to discuss properties of the noncommutative Koszul (Golod-Shafarevich) complex
and their implications for the structure  of the representation space. Namely, we consider situation when the Golod-Shafarevich complex provides a DG resolution of the algebra. We show that the Jordan plane is an example illustrating the fact that RCI (representational complete intersections), introduced by Ginsburg and Etingof in \cite{GE} always have this property of the Golod-Shafarevich complex.

Let us present first a construction, appeared in \cite{GS}, of the complex associated to an algebra, presented by generators and relations. This complex could be considered as an analogue of the Koszul complex in commutative algebra. Let $A=k\langle x_1,...,x_d \rangle / \{ f_1,...,f_m \}$. Denote the free algebra $k\langle x_1,...,x_d \rangle $ by $R_d$, and an ideal in $R_d$ generated by polynomials
$ f_1,...,f_m$ by $I$. We suppose that  $f_1,...,f_m$ are homogeneous, hence $A$ inherits the grading by the degree on variables  $ x_1,...,x_d$ form the free algebra $R_d$.

Let us denote now by $M=k\langle x_1,...,x_d, u_1,...,u_m \rangle $, where variables $u_i$ are in one to one correspondence with polynomials $ f_1,...,f_m$. Linear basis of $M$ over $k$ consists obviously of
monomials
$$ a_0u_{i_1}a_1u_{i_2}...a_{k-1}u_{i_k}a_k,$$
where $a_i$ are monomials from $R_d$. {\it Homological degree} of such a monomial is its degree on variables $u_i$. Obviously, if $M^{(k)}$ is a linear combination of monomials of homological degree $k$, $M$ becoming a graded algebra: $M= \sum\limits_{k=0}^{\infty} M^{(k)}$.

We can define the following differential on it.
Let $\delta (x_i)=0, \delta (u_i)=f_i$, and it is extended as a graded derivation, i.e. by the rule:

$$\delta(vw)=\delta(v)w+(-1)^{{\rm deg} v}  v \delta w.$$

Defined in such a way differential have the following formula:

$$\delta ( a_0u_{i_1}a_1u_{i_2}...a_{k-1}u_{i_k}a_k) = \sum\limits_{l=1}^k (-1)^l a_0u_{i_1}...a_{l-1}f_{i_l}a_{l}...u_{i_k}a_k.$$

It is easy to see that  $\delta ^2=0$, and $\delta$ drops the homological degree by one, so we have a complex of $R_d$-modules

$$ 0 \strl{\delta} M^{(0)} \strl{\delta} M^{(1)} \strl{\delta} M^{(2)} \strl{\delta} ... $$

It is clear that $H_0(M) = A$, so we can rewrite it as an $A$ - resolution, which is in general not acyclic:

$$ 0 \leftarrow A \leftarrow M^{(0)} \strl{\delta} M^{(1)} \strl{\delta} M^{(2)} \strl{\delta} ... $$

it is exact in terms $A$ and $M^{(0)}$, but  not in the other terms.

This complex we denote by $Sh_{\bullet}(A)$.

\begin{definition} In case the complex $Sh_{\bullet}(A)$ is a DG-algebra resolution of  $A$, that is all
$H^{(i)}(M)=0$, for $i \geq 1$, we say that $A$ is a noncommutative complete intersection (NCCI).
\end{definition}

Now we give two other equivalent definitions of NCCI \cite{GS}.

\begin{theorem}
The following are equivalent:

\begin{itemize}

\item[(i)]The noncommutative Koszul complex $Sh_{\bullet} A$ is
a DG algebra resolution of A, that is

$H_0(Sh_{\bullet} A, \delta)= A$, $ H_n(Sh_{\bullet} A, \delta)= 0$, $n > 0$.

\item[(ii)] The Hilbert series of $A$  coincides with the Golod-Shafarevich series:

$H_A= (1-nt+dt^2)^{-1}$.

 \item[(iii)] For any right $A$-module $M$,

 $\rm{Ext}^i_A (A/A_+ , M)=0$, for $i \geq 3,
({\rm gl.dim} A \leq 2).$

\end{itemize}

\end{theorem}

In particular, from the condition (iii) one can see, that  NCCI implies Koszulity:
put $M=k_A$.

This provides us with another way to ensure Koszulity of the Jordan plane $R$. As we noticed in Lemma \ref{14.1}, the defining relations of $R$  form a Gr\"obner basis. Using this it is easy to calculate
the Hilbert series of $R$: $H_R=(1-t)^{-2}$, so by (ii) it is a NCCI, hence Koszul.

\vspace{3mm}

Now we turn to the properties of representation spaces, which should be a reflection of NCCI.
The notion of {\it representational complete intersection } (RCI) was introduced by Ginzburg and Etingof \cite{GE}

\begin{definition} An algebra is an RCI if for infinitely many $n$ the spaces of $n$-dimensional representations of that algebra are a (commutative) complete intersections.

\end{definition}

The main source of examples of RCI is provided by the preprojective algebras of finite quivers.
It was also shown in the above paper that RCI implies NCCI.
Answering the question of Etingof {\cite{E}} we provide here one more example of RCI (which is not a preprojective algebra).

\begin{theorem} The Jordan plane $R=k\langle x,y | xy-yx=y^2 \rangle$ is a RCI.

\end{theorem}

\begin{proof} We have proved in Lemma\ref{dimstr} and Theorem\ref{m} that the representation variety $mod(R,n)$ of the Jordan plane $R$ is an equidimensional variety with each irreducible component having the dimension $n^2$. This means that the dimension of the variety is equal to the dimension of the whole space of pairs of matrices, $2n^2$, minus the dimension of the space of defining relations, which is $n^2$, so all those varieties are complete intersections in the classical (commutative) sense.

\end{proof}

\section{Acknowledgments}

 The author would like to acknowledge the support and hospitality of the International Erwin Schr\"odinger Institute in Vienna
where the results were presented during the special  program on "Golod-Shafarevich groups and algebras and rank gradient".

The work on this circle of questions was started
during my visit at the
Max-Planck-Intitut f\"ur Mathematik in Bonn. The content of the first 9 chapters appeared as an MPI preprint \cite{MP}. I am thankful to this
institution for support and hospitality and  to
many colleagues with whom ideas related to
this paper have been discussed.

\end{document}